\documentclass[11pt]{article}
\usepackage{amsthm, amsmath, amssymb, amsfonts, url, booktabs, setspace, fancyhdr, bm}
\usepackage{pgf,tikz,tkz-graph}
\usetikzlibrary{math}
\usepackage{geometry}
\usepackage{hyperref, enumerate}
\usepackage[shortlabels]{enumitem}
\usepackage[babel]{microtype}
\usepackage[english]{babel}
\usepackage[capitalise]{cleveref}
\usepackage{bbm}
\usepackage{csquotes}
\usepackage{mathrsfs}
\usepackage{mathabx}
\usepackage{float}

\usepackage[normalem]{ulem}

\usepackage[T1]{fontenc}
\usepackage{amsfonts}
\usepackage{amscd}
\usepackage{graphicx}
\usepackage{enumitem}
\usepackage{verbatim}
\usepackage{amsmath,caption}
\usepackage{pdfpages,xcolor,framed,color}
\usepackage{todonotes}

\PassOptionsToPackage{hyphens}{url}

\newtheorem{thr}{Theorem}[section]
\newtheorem{ex}[thr]{Example}
\newtheorem{q}[thr]{Question}

\newtheorem{lem}[thr]{Lemma}
\newtheorem{prop}[thr]{Proposition}

\newtheorem{cor}[thr]{Corollary}
\theoremstyle{definition}
\newtheorem{defi}[thr]{Definition}

\newtheorem*{rem}{Remark}
\newtheorem{examp}[thr]{Example}

\newcommand*{\myproofname}{Proof}

\newcommand*{\sH}{\mathscr{H}}
\let\le\leqslant
\let\ge\geqslant
\let\leq\leqslant

\let\geq\geqslant

\usepackage{bookmark}

\newcommand{\abs}[1]{\left|#1\right|}

\renewcommand*{\phi}{\varphi}

\def\C{\mathcal{C}}
\def\D{\mathcal{D}}

\title{Fractional list packing for layered graphs}

\author{
	Stijn Cambie
 \thanks{Department of Computer Science, KU Leuven Campus Kulak-Kortrijk, 8500 Kortrijk, Belgium. 
 Supported by a FWO grant with grant number 1225224N. Email: \protect\href{mailto:stijn.cambie@hotmail.com}{\protect\nolinkurl{stijn.cambie@hotmail.com}}.}	
	\and
	Wouter Cames van Batenburg%
	\thanks{
 D\'epartement d'Informatique, Universit\'e libre de Bruxelles, Belgium. Supported by the Belgian National Fund for Scientific Research (FNRS).
		Email: \protect\href{mailto:w.p.s.camesvanbatenburg@gmail.com}{\protect\nolinkurl{w.p.s.camesvanbatenburg@gmail.com}}.}}

\date{\today}

\begin{document}
\maketitle
\begin{abstract}
The \emph{fractional list packing number} $\chi_{\ell}^{\bullet}(G)$ of a graph $G$ is a graph invariant that has recently arisen from the study of disjoint list-colourings. It measures how large the lists of a list-assignment $L:V(G)\rightarrow 2^{\mathbb{N}}$ need to be to ensure the existence of a `perfectly balanced' probability distribution on proper $L$-colourings, i.e., such that at every vertex $v$, every colour appears with equal probability $1/|L(v)|$.
In this work we give various bounds on $\chi_{\ell}^{\bullet}(G)$, which admit strengthenings for correspondence and local-degree versions.
As a corollary, we improve theorems on the related notion of flexible list colouring. 
In particular we study Cartesian products and $d$-degenerate graphs, and we prove that $\chi_{\ell}^{\bullet}(G)$ is bounded from above by the pathwidth of $G$ plus one. 
The correspondence analogue of the latter is false for treewidth instead of pathwidth.
\end{abstract}

\section{Introduction}

In classical graph colouring, the goal is to find a vertex-colouring $c:V(G) \rightarrow \mathbb{N}$ of a graph $G$, so that adjacent vertices receive distinct colours. Such a colouring is called \emph{proper}.
If there exists a proper colouring for which each vertex maps to an element of $[k]:=\{1,\ldots,k\}$, then $G$ is \emph{$k$-colourable}, and the \emph{chromatic number} $\chi(G)$ is the smallest $k$ for which this is possible. 
A well-established generalisation is \emph{list colouring}, introduced in the seventies independently by~\cite{Viz76} and~\cite{ERT80}. In this setting, each vertex $v$ is assigned a private list $L(v) \subseteq \mathbb{N}$, and the goal is again to find a proper colouring $c:V(G) \rightarrow \mathbb{N}$, but now with the more local condition that $c(v)\in L(v)$ for every $v$. Such a colouring is an \emph{$L$-colouring}. A list-assignment $L$ is \emph{$k$-fold} if each list $L(v)$ has size $k$, and a graph $G$ is \emph{$k$-choosable} if it is $L$-colourable for every $k$-fold $L$. The \emph{list-chromatic number} $\chi_{\ell}(G)$ (sometimes called \emph{choosability}) of $G$ is the smallest $k$ such that $G$ is $k$-choosable. Since the list-assignment $L$ that has the same list $[k]$ at each vertex needs to admit an $L$-colouring, it follows that $k$-choosable implies $k$-colourable, and hence $\chi(G)\leq \chi_{\ell}(G)$. Perhaps surprisingly, the opposite inequality is very false, because there exist graphs with $\chi(G)=2$ and $\chi_{\ell}(G)$ arbitrarily large. 
Besides being interesting in its own right, a major benefit of list colouring has been its versatility. Since not all lists need to have the same size, they allow for efficient `algorithmic' induction proofs that would not be possible with classic colouring, a good example being Thomassen's proof~\cite{Thomassen94} that planar graphs are $5$-choosable.

In the 2010s, a similar development took place, when Dvo\v{r}\'{a}k and Postle introduced~\cite{DvPo18} the \emph{correspondence chromatic number} $\chi_c(G)$, which helped them prove results on planar graphs that were out of reach with classical and list colouring. While list colouring had liberated the lists from all being the same, correspondence colouring decentralises the notion of being `proper', by allowing edges to choose their own pairs of forbidden adjacent colours. We refer to section~\ref{sec:defnot} for the precise definition, but for now it suffices to know that $\chi_{\ell}(G)\leq \chi_c(G)$ always.\\

Everything discussed so far concerns finding \emph{at least one} colouring of some type. It is natural to ask for more. For instance simply try to count the number of colourings; if there exists one, often there exist many~\cite{Thomassen07,Thomassen09, Rosenfeld20}.

While helpful, there is a risk that this sketches us a false image of the graph, as it could contain a tiny part that accomodates many colourings while the majority of the graph remains rigid. An alternative fruitful research direction~\cite{Tuza97survey, ALBERTSON98, AM99, DL17, EGHKPS18} 
has been to see under what circumstances one can precolour some vertices arbitrarily, and still be able to extend it to a proper colouring. This already offers a more holistic view of the graph. One downside is that not every vertex can be arbitrarily precoloured as, for instance, colouring all vertices red is not extendable to a proper colouring.
This led Dvo\v{r}\'{a}k, Norin and Postle~\cite{DNP19} to introduce the notion of \emph{(weighted) flexibility}, which allows all vertices to be precoloured (yielding a colouring that is not necessarily proper), but then as a trade-off merely imposes that for a large fraction of the vertices their initial colour survives in the final proper colouring.

Formally, a graph $G$ is called \emph{$\epsilon$-flexible} with respect to a list-assignment $L$ if for every vertex subset $D\subseteq V(G)$ and every collection of `precolouring' requests $\{r(v) \in L(v) \mid v\in D\}$, there is a proper $L$-colouring $c$ of $G$ such that $c(v)=r(v)$ for at least $\epsilon \cdot |D|$ vertices of $D$.

In the paper where they introduced it, Dvo\v{r}\'{a}k, Norin and Postle already noted that this somewhat unwieldy notion is implied by an elegant stronger probabilistic property. A graph is \emph{weighted $\epsilon$-flexible} with respect to a list-assignment $L$ if there exists a probability distribution on $L$-colourings $c$ such that for every vertex $v$ and every colour $x \in L(v)$, the probability that $c$ assigns colour $x$ to $v$ is at least $\epsilon$. A graph is (\emph{weighted}) \emph{$\epsilon$-flexibly $k$-choosable} if it is (weighted) $\epsilon$-flexible with respect to every $k$-fold list-assignment $L$. 

\begin{lem}[\cite{DNP19}]
Every weighted $\epsilon$-flexibly $k$-choosable graph is $\epsilon$-flexibly $k$-choosable. The converse is false.
\end{lem}

Clearly, the larger the value of $\epsilon$, the stronger the property.
However, it is not possible to be weighted $\epsilon$-flexibly $k$-choosable for a value $\epsilon> \frac{1}{k}$, for otherwise the expected number of colours at a given vertex would be larger than one. Therefore the holy grail is to prove that a graph is weighted $\frac{1}{k}$-flexibly $k$-choosable. The smallest integer $k$ for which this holds is the \emph{fractional list packing number}, denoted $\chi_{\ell}^{\bullet}(G)$. The reason for the seemingly unrelated name is that it has independently arisen~\cite{CCDK23} from the study of disjoint list-colourings~\cite{CCDK21,CCZ23,M23,CS-R24+, CH23, CDKH24, CT24, KM24, KMMP22} and the associated \emph{list packing number} $\chi_{\ell}^{\star}(G)$. The latter is easiest to interpret as the minimum size $k$ of the lists such that there exist $k$ $L$-colourings that simultaneously partition each list, but it can equivalently be defined in terms of the chromatic numbers of certain auxiliary \emph{cover graphs} of $G$. 
In a similar manner, $\chi_{\ell}^{\bullet}(G)$ is determined by the fractional chromatic number of those same cover graphs so that $\chi(G)\leq \chi_{\ell}(G)\leq \chi_{\ell}^{\bullet}(G)\leq \chi_{\ell}^{\star}(G)$ holds for every graph $G$. We refer the reader to~\cite{CCDK23} for formal details about this connection. The main thing to realise here is that $\chi_{\ell}^{\bullet}(G)$ measures a robust and balanced list-chromatic structure of the graph.
Indeed, $\chi_{\ell}^{\bullet}(G) \leq k$ if and only if: for every list-assignment $L$ with every list of size at least $\lfloor k \rfloor$ there exists a probability distribution on the $L$-colourings of $G$ such that at \emph{every} vertex, \emph{every} colour from its list appears with the \emph{same} probability $\frac{1}{|L(v)|}$. 
By taking correspondence colouring instead of list colouring in the definition (see Section~\ref{sec:defnot} for the details), one can also define the \emph{correspondence packing number} $\chi_{c}^{\star}(G)$ and the \emph{fractional correspondence packing number} $\chi_c^{\bullet}(G)$, which are more convenient to work with. Every graph $G$ satisfies $\chi_c(G)\leq \chi_c^{\bullet}(G) \leq \chi_c^{\star}(G)$ and $\chi_{\ell}^{\bullet}(G)\leq \chi_c^{\bullet}(G)$. Hence the main focus in this work will be on upper bounding $\chi_c^{\bullet}(G)$, and we will only refer to $\chi_c^{\star}(G)$ to illustrate the challenges and strength of our results. For the reader who quickly skims through this introduction, we emphasise:
$$\text{If }\chi_c^{\bullet}(G)\leq k \text{ for some integer }k, \text{ then } G \text{ is weighted }\frac{1}{k}\text{-flexibly } k\text{-choosable.}$$

Having introduced the colouring numbers of interest, we now turn to the layered types of graphs for which we wish to bound them, and why. A graph $G$ is \emph{$d$-degenerate} if there exists an ordering of the vertices $v_1,v_2,\ldots, v_n$ such that each $v_i$ has at most $d$ neighbours $v_j$ with $j>i.$ This ordering represents an ordered layering of the graph (each layer consisting of one vertex) which is useful because it enables a greedy algorithmic colouring procedure. Deleting the vertex $v_1$ with lowest index, by induction one can find a proper $(d+1)$-colouring of $G-v_1$, which we can then extend to a proper $(d+1)$-colouring of $G$ because at least one colour does not appear on the neighbours of $v_1$. The simple procedure also works for list colouring and correspondence colouring, and is the canonical reason why $\chi, \chi_{\ell}$ and $\chi_{c}$ are all bounded by one plus the \emph{maximum degree} $\Delta(G)$.
However, this algorithmic intuition fails for the (fractional) packing numbers! Here are a few examples. 
\begin{enumerate}[(i)]
 \item Complete graphs $K_n$ on an odd number of vertices $n$ satisfy $\chi_c^{\star}\geq n+1$. Yuster~\cite{Yus21} amalgamated a conjecture going back to various authors including Fischer, K\"uhn, Osthus and Catlin, that this is optimal. The state-of-the-art upper bound is $(1.78 +o(1)) \cdot n$.
 \item For all $d\geq 1$, there are bipartite $d$-degenerate graphs with $\chi_c^{\star}=2d$. There is no graph with $\chi_c^{\star}=3$. See~\cite{CCDK21, CH23}.
 \item For all $d\geq 2$ there are $d$-degenerate graphs with $\chi_{\ell}^{\bullet}\geq d+2$. Both for $\chi_{\ell}^{\bullet}$ and $\chi_c^{\bullet}$ the optimal upper bound in terms of $d$ is unknown; it could be anywhere between $d+2$ and $2d$. See~\cite{CCDK23}.
\end{enumerate}
These examples illustrate that deleting a single vertex (or edge) could reduce $\chi_c^{\star}, \chi_c^{\bullet}, \chi_{\ell}^{\bullet}$ by \emph{more than one}, and a priori it is not obvious when and where such jumps may appear. It makes an algorithmic approach especially daunting.\\

While a degeneracy ordering is thus not universally adequate for (fractional) packing numbers, in this work we identify a type of layering that does always work for $\chi_{c}^{\bullet}$. It maintains the possibility to construct the fractional packing algorithmically along the layers. Concretely: building upon our previous work~\cite{CCZ23} we observe that it is sufficient to subdivide the graph into an ordered collection of layers such that each layer induces a graph with lower value of $\chi_c^{\bullet}$, and such that each vertex has at most one neighbour in the union of the previous layers. We make this precise in~\cref{lem:chicwidth} via the concept of \emph{$\chi_c^{\bullet}$-width}. 
From this Lemma we immediately obtain a number of tight corollaries for graphs that naturally have a layered structure, including bounded-degree graphs, planar graphs, \emph{Cartesian products}, as well as graphs of bounded \emph{treedepth}.\\

Treedepth (denoted $\mathrm{td}$), treewidth (denoted $\mathrm{tw}$) and pathwidth (denoted $\mathrm{pw}$) are part of a suit of graph invariants that measure how close a graph is to being a tree (in the former two cases) or a path (in the third case). We refer to~\cref{sec:defnot} for their precise definitions. They feature prominently in the algorithmic literature because combinatorial problems that are NP-hard in general tend to be in P when restricted to graphs of bounded $\mathrm{td}, \mathrm{tw}$ or $\mathrm{pw}$. See e.g.~\cite{Courcelle90, Noble98, BST23}.
This is essentially due to the fact that they can be decomposed into small layers that are structured in a tree-like fashion. The relevant problem can be solved efficiently on each layer, and the tree-like structure allows for efficient combination of these solutions.
It is well-known that every graph satisfies 
$$\mathrm{tw} \leq \mathrm{pw} \leq \mathrm{td} -1.$$ 
Moreover, $d\leq \mathrm{tw}$ for every $d$-degenerate graph. As alluded to above, finding the optimal bounds on $\chi_{\ell}^{\bullet}$ or $\chi_c^{\bullet}$ in terms of $d$ is a major challenge and defies the greedy procedure; this motivated us to prove $\chi_c^{\bullet}\leq \mathrm{td} $ instead. 
Having succeeded in that, we then went one step further and derived the strengthening $\chi_c^{\bullet} \leq \mathrm{pw} +1$.
In~\cref{sec:conclusion} we briefly discuss why the situation for treewidth is different.
The state-of-the-art there is essentially the same as for degeneracy: there exist graphs with $\chi_c^{\bullet}= \mathrm{tw}+2$, but in general we only know it is $\leq 2\cdot \mathrm{tw}$.

\section{Main results}
Bradshaw, Masa\v{r}\'{\i}k and Stacho proved:
\begin{thr}[\cite{BMS22}]
Let $G$ be a graph of treedepth $k$. Then $G$ is $\frac{1}{k}$-flexibly $k$-choosable, and hence weighted $\frac{1}{k^2}$-flexibly $k$-choosable. 
\end{thr}

We improve upon this via an application of~\cref{lem:chicwidth}, which is our general layering lemma. We obtain: 
\begin{thr}\label{thm:treedepth}
Let $G$ be a graph of treedepth $k$. Then $\chi_c^{\bullet}(G)\leq k$, and hence $G$ is weighted $\frac{1}{k}$-flexibly $k$-choosable. 
\end{thr}

Using a more specialised approach, we were then able to strengthen~\cref{thm:treedepth} to pathwidth. Recall that for every graph, its pathwidth is at most its treedepth minus one.
\begin{thr}\label{thm:pathwidth}
Let $G$ be a graph of pathwidth $k$. Then $\chi_c^{\bullet}(G)\leq k+1$.
\end{thr}

Both Theorems~\ref{thm:treedepth} and~\ref{thm:pathwidth} are sharp due to e.g. the clique $K_k$ and complete bipartite graphs $K_{k-1,t}$, for sufficiently large $t$. Both graphs have treedepth $k$ and pathwidth $k-1$ and $\chi_{\ell}=\chi_c^{\bullet}=k$. In contrast, nor~\cref{thm:treedepth} nor~\cref{thm:pathwidth} can be extended to the correspondence packing number, as $\chi_c^{\star}(K_{k-1,t})=2k-2$ for all $t$ sufficiently large~\cite[Cor. 34]{CCDK21}, a multiplicative factor two higher! \\

Another canonical type of layered graphs are those formed by the Cartesian product $G_1 \square G_2$ of two graphs $G_1, G_2$. It is easy to see~\cite{sabidussi57} that $\chi(G_1 \square G_2)=\max(\chi(G_1),\chi(G_2))$. However, Borowiecki, Jendrol, Kr\'al and Mi\v{s}kuf showed that this does not generalise to list colouring. They proved the following optimal result, which was later extended to correspondence colouring by Kaul, Mudrock, Sharma and Stratton~\cite{KMSS23}:
\begin{thr}[\cite{BJKM06}]\label{thm:cartesianproductList}
Let $G_1$ and $G_2$ be graphs, the latter being $d$-degenerate. Then their Cartesian product satisfies
$$\chi_{\ell}(G_1 \square G_2) \leq \chi_{\ell}(G_1)+d.$$
Moreover, for every graph $G_1$ and every value of $d$, some $d$-degenerate graph $G_2$ attains the bound.
\end{thr}

In the context of list- and correspondence colouring, bounds in terms of maximum degree (denoted $\Delta(G)$) are often a direct corollary of the same bound with maximum degree replaced with degeneracy. This is no longer true for fractional list- and correspondence packing. 
In particular, it has been shown~\cite{CCDK23} that the greedy bound 
$\chi_{c}^{\bullet}(G) \leq \Delta(G)+1$ still survives, despite the existence of $d$-degenerate graphs $G$ with $\chi_{\ell}^{\bullet}(G)\geq d+2$. It is unknown, and seems a challenging task, to determine what are the best-possible upper bounds in terms of $d$. It could be anywhere between $d+2$ and $2d$. Thus it is no surprise that in our efforts to generalise~\cref{thm:cartesianproductList}, for now we have to be satisfied with an analogue in terms of maximum degree.

\begin{thr}\label{thm:cartesianproductFractPack}
Let $G_1$ and $G_2$ be graphs. Then their Cartesian product satisfies
$$\chi_{\ell}^{\bullet}(G_1 \square G_2) \leq \chi_{\ell}^{\bullet}(G_1)+\Delta(G_2),$$
and
$$\chi_{c}^{\bullet}(G_1 \square G_2) \leq \chi_{c}^{\bullet}(G_1)+\Delta(G_2).$$
\end{thr}

To put this further in perspective, we stress that even for the Cartesian product of two cliques the situation is unclear. Since $K_n \square K_m$ is the linegraph of $K_{n,m}$, a celebrated result of Galvin~\cite{Gal95} yields $\chi_{\ell}(K_n \square K_m)= \max(n,m)$. For correspondence colouring the exact values are unknown. Kostochka and Bernshteyn~\cite{BK19} showed that $\chi_c(K_n \square K_m)\geq n+1$ if $n=m\geq 2$, while in the asymptotic regime Postle and Molloy~\cite{MP22} proved that it is at most $(1+o(1)) \max(n,m)$. For $\chi_c^{\bullet}$ the best general upper bound we are aware of is $n+m-1$, which is also implied by~\cref{thm:cartesianproductFractPack}.\\

Despite the difficulty of proving $\chi_{c}^{\bullet}(G) \leq d+2$ for $d-$degenerate graphs (if true!), we are able to make some modest progress.
Kaul, Mathew, Mudrock and Pelsmajer proved:
\begin{thr}[\cite{KMMP22}]
Let $G$ be a $d$-degenerate graph. Then $G$ is $\frac{1}{2^{d+1}}$-flexibly $(d+2)$-choosable.
\end{thr}

They proved this via a probabilistic statement which is reminiscent of \emph{weighted} $\frac{1}{2^{d+1}}$-flexibility, but is still quite different. For each collection of requests they constructed a probability distribution on list-colourings such that at every vertex, the requested colour appears with probability at least $\frac{1}{2^{d+1}}$. For weighted $\frac{1}{2^{d+1}}$-flexibility however we need a uniform probability distribution that works for all collections of requests simultaneously. This is what we prove.
In fact we derive a local-degree strengthening which also works for correspondence colouring, and in particular implies:
\begin{thr}\label{thm:wflex_ddegenerate}
Let $G$ be a $d$-degenerate graph. Then $G$ is weighted $\frac{1}{2^{d+1}}$-flexibly $(d+2)$-choosable.
\end{thr}

\textbf{Outline}\\
The results on treedepth (\cref{thm:treedepth}) and Cartesian products (among which \cref{thm:cartesianproductFractPack}) are in~\cref{sec:cartesianproducts}, the proof of~\cref{thm:pathwidth} on pathwidth is in~\cref{sec:pw+1}, and we prove~\cref{thm:wflex_ddegenerate} in~\cref{sec:degeneracy}.

\section{Definitions, notation and technical tools}\label{sec:defnot}

\begin{defi}
The \emph{Cartesian product} $A \square B$ of two graphs $A$ and $B$ is the graph with vertex set $V(A) \times V(B)$, such that two vertices $(a_1,b_1)$ and $(a_2,b_2)$ are adjacent in $A \square B$ if and only if either $a_1=a_2$ and $b_1b_2\in E(B)$, or $b_1=b_2$ and $a_1a_2 \in E(A)$.
\end{defi}

\begin{defi}
Treedepth, treewidth and pathwidth have many equivalent definitions. Here we describe the ones that we use in this manuscript.
 The \emph{treedepth} $\mathrm{td}(G)$ of a connected graph $G$ is the minimum integer $t$ such that some supergraph of $G$ has a spanning rooted tree $T$ (also called a \emph{Tr\'emaux tree}) of height $t$ such that every edge of $G$ joins an ancestor-descendant pair of $T$. 
 If $G$ is not connected, the treedepth is the maximum of $\mathrm{td}$ over all its connected components. The \emph{treewidth} $\mathrm{tw}(G)$ of a graph $G$ is the smallest integer $t$ such that $G$ is a subgraph of some $t$-tree. A \emph{$t$-tree} is a graph that can be constructed via the following procedure. The complete graph $K_{t+1}$ is a $t$-tree, and: given any $t$-tree $T$, creating a new vertex and joining it with a copy of $K_t$ in $T$ is also a $t$-tree. 
 A $t$-tree which is either $K_{t+1}$ or has exactly two vertices of degree $t$, is a \emph{$t$-path}.
A \emph{$t$-clique-separator} of a $t$-path $P$ is a vertex subset $B$ of $P$ that induces a clique on $t$ vertices such that the graph $P-B$ is disconnected.
A \emph{t}-caterpillar is a $t$-tree that can be partitioned into a $t$-path, and a (possibly empty) set of \emph{pendant} degree-$t$-vertices that are each adjacent to some $t$-clique-separator of the $t$-path.
Finally, the \emph{pathwidth} $\mathrm{pw}(G)$ of $G$ is the smallest $t$ such that $G$ is a subgraph of some $t$-caterpillar.\\

The following is a more elaborate description of $t$-caterpillars, which will be helpful to understand our proofs in~\cref{sec:pw+1}. Initialise an \emph{active} vertex set $A_1=\{v_1,v_2,\ldots,v_{t+1}\}$ and let $G_1$ be the $t$-caterpillar on $A_1$ which is isomorphic to $K_{t+1}$. Now let $m\geq 1$. Suppose we are given a $t$-caterpillar $G_m$ and the active vertex set $A_m \subset V(G_m)$ of size $t+1$. Let $B_m\subset A_m$ be any subset of size $t$, create a new vertex $v_{t+1+m}$, make $A_{m+1}:= B_m \cup \{v_{t+1+m}\}$ the new active vertex set, and let $Z$ be the graph on $A_{m+1}$ which is isomorphic to $K_{t+1}$. Then $G_{m+1}:=G_m \cup Z$ is also a $t$-caterpillar. 
We can think of the set $B_i$ as the interface along which a new active clique $A_{i+1}$ is glued on the old active clique $A_i$. Note that the unique vertex in $A_{i}\setminus B_i$ cannot appear in any active set $A_j$ with $j>i$ anymore. If $B_{i+1}\neq B_i$ for every $i$, then the produced t-caterpillar is a $t$-path. 
If $B_{i+1}=B_i$, then the vertex $v_{t+1+i}$ can only be part of one active set (namely $A_{i+1}$), so in the final $t$-caterpillar $v_{t+1+i}$ must be a \emph{pendant} vertex of degree $t$, which is incident to the $t$-clique separator $B_i$. The remaining non-pendant vertices of $G$ induce a $t$-path. 
The order in which the vertices of a $t$-caterpillar $G$ are added defines a vertex ordering $v_1,v_2,\ldots,v_{|V(G)|}$. Naturally one can also construct the $t$-caterpillar in the opposite direction, starting from $v_{|V(G)|}$ instead of $v_1$. This reverses the ordering of the active sets $A_i$ and $t$-clique-separators $B_i$, but it might produce a vertex ordering that is \emph{different} from $v_{|V(G)|}, \ldots, v_2,v_1$. We need this alternative vertex ordering in the proof of~\cref{thr:chicbullet<pw+1}.
\end{defi}

In the remainder of this section we give the precise definitions for correspondence colouring, fractional packing, and the graph invariants $\chi_{\ell}, \chi_{\ell}^{\bullet}, \chi_{\ell}^{\star}$, and $\chi_c, \chi_c^{\bullet}, \chi_{c}^{\star}$. We also summarise tools for fractional packing that were proved in~\cite{CCZ23}.\\

\begin{defi}
Given a graph $G$, a pair $\sH=(L,H)$ is a \emph{correspondence-cover} of $G$ if $H$ is a graph and $L:V(G)\to 2^{V(H)}$ is a mapping that satisfies the following:
\begin{enumerate}[(i)]
 \item $L$ induces a partition of $V(H)$,
 \item the bipartite subgraph of $H$ induced between $L(u)$ and $L(v)$ is empty whenever $uv\notin E(G)$,
 \item\label{itm:corrdef} the bipartite subgraph of $H$ induced between $L(u)$ and $L(v)$ is a matching $M_{uv}$ whenever $uv\in E(G)$, 
 \item the subgraph of $H$ induced by $L(v)$ is a clique for each $v\in V(G)$.
\end{enumerate}
\end{defi}

It can be convenient to drop $\sH$ and $L$ from the notation, saying for example that some property of the correspondence-cover holds if it holds for the \emph{cover graph} $H$.
To see how this concept is related to list colouring, observe that a list-assignment $L$ of $G$ naturally gives rise to a correspondence-cover by choosing the matching between $L(u)$ and $L(v)$ such that same colours are joined by an edge, for every edge $uv$ of $G$. In that case $\sH$ is a \emph{list-cover}. While not every correspondence-cover is a list-cover, with some abuse of notation we will always refer to the vertex sets $L(v) \subseteq V(H)$ as \emph{lists}. A correspondence-cover is \emph{$k$-fold} if $|L(v)|=k$ for each vertex $v$ of $G$.
In this case, we can write $L(v)=\{1_v,2_v,\ldots,k_v\}$, and when no confusion arises $L(v)=[k]=\{1,2,\ldots,k\}$, for every $v \in V(G)$.
An edge subset $A\subseteq E(G)$ is assigned \emph{identity matchings} if for every $uv\in A$, the matching in the cover graph connects $x_u$ and $x_v$ for every $x \in [k].$

Given a correspondence-cover $\sH=(L,H)$, an \emph{independent transversal} of $\sH$ is an independent set $I$ of $H$ such that $|I \cap L(v)|=1$ for every vertex $v$ of $G$, i.e., such that every list is intersected precisely once. In this terminology, the \emph{correspondence chromatic number} $\chi_c(G)$ of $G$ is the least $k$ such that every $k$-fold correspondence-cover of $G$ has an independent transversal. 
Similarly, the list-chromatic number $\chi_{\ell}(G)$ can be reinterpreted as the least $k$ such that every $k$-fold list-cover of $G$ has an independent transversal, since in that case the independent transversals of $\sH$ are in bijection with the list-colourings of $G$.
As is common in the literature, for an independent transversal $I$ of $(H,L)$, the function $c:V(G) \rightarrow V(H)$ that sends each vertex $v$ to $L(v)\cap I$ is referred to as a \emph{correspondence-colouring} of $G$. Note that the independent transversals and correspondence-colourings are in bijection with each other.\\

The following definition is core to the fractional packing numbers $\chi_c^{\bullet}$ and $\chi_{\ell}^{\bullet}$ that we study in this work. We stress that not all lists need to have the same size.
\begin{defi}
A correspondence-cover $\sH=(L,H)$ of a graph $G$ admits a \emph{fractional packing} if there exists a probability distribution $\mathbb{P}$ on independent transversals of $\sH$ such that the following holds for a random independent transversal $I$ taken from that distribution:
$$\mathbb{P}(x\in I)\geq \frac{1}{|L(v)|} \text{, for every } v\in V(G) \text{ and every } x\in L(v).$$
\end{defi}

\begin{defi}
The \emph{fractional correspondence packing number} $\chi_c^{\bullet}(G)$ of a graph $G$ is the least integer $k$ such that every $k$-fold correspondence-cover of $G$ admits a fractional packing.
Similarly, the \emph{fractional list packing number} $\chi_{\ell}^{\bullet}(G)$ is the least $k$ such that this holds for every $k$-fold list-cover.
\end{defi}

Note that if all lists have the same size $k$, then $(L,H)$ admitting a fractional packing is equivalent to $H$ having \emph{fractional chromatic number} $k$ (one direction is immediate, the other follows from~\cite[Prop.~11]{CCDK23}). This is what motivated the definition of the fractional packing numbers. It also means that for deriving a lower bound $\chi_c^{\bullet}(G)>k$, it suffices to prove that some subgraph of some $k$-fold covergraph $H$ of $G$ has fractional chromatic number $>k$.\\

For upper bounds, the next two technical lemmas from~\cite{CCZ23} are of use to us. The first,~\cite[Lem.~7.1]{CCZ23}, says that the existence of a fractional packing is preserved under increasing the size of a list.

\begin{lem}[\cite{CCZ23}]\label{lem:monotonicityfractionalpacking}
Let $G$ be a graph and let $s: V(G) \rightarrow \mathbb{N}$ be a function. Suppose every correspondence-cover $\sH^{-}=(L^{-},H^{-})$ of $G$ with $|L^{-}(v)| = s(v)$ for all $v$ has a fractional packing. 
Then also every correspondence-cover $\sH=(L,H)$ of $G$ with $|L(v)| \geq s(v)$ for all $v$ has a fractional packing. The same holds mutatis mutandis for list-covers.
\end{lem}

For a subgraph $T$ of $G$, the graph obtained from $G$ by deleting the vertices of $T$ is denoted by $G-V(T)$. 
For a vertex $u$ of $T$, a \emph{neighbour outside $T$} refers to a neighbour of $u$ with respect to the graph $G$ that is in $V(G)\setminus V(T)$.
For a correspondence-cover $\sH=(L,H)$ of $G$, its \emph{restriction} to $G-V(T)$ is the cover of $G-V(T)$ obtained by removing $\bigcup_{v\in T} L(v)$ from $H$. 
While the next lemma is rather technical, its core message is this: to inductively upper-bound the fractional correspondence packing number of a graph, it suffices to find an induced subgraph that has smaller fractional packing number and few neighbours in the remainder of the graph. We remark that~\cite[Lem.~7.3]{CCZ23} was only stated for correspondence-covers, but the same proof works verbatim for list-covers.

\begin{lem}\cite{CCZ23}]\label{lem:technicalfractionallemma}
Let $G$ be a graph with a correspondence-cover $\sH=(L,H)$. If $G$ contains an induced subgraph $T$ satisfying the following four conditions, then $\sH$ has a fractional packing.

\begin{enumerate}[(i)]
 \item \label{item:TechInd1} Every vertex of $T$ has at most one neighbour outside $T$;
 \item \label{item:TechInd2} $2 \leq |L(u)|\leq |L(v)|$ for every $u \in V(T)$ with one neighbour $v$ outside $T$.
 \item \label{item:TechInd3} The restriction of $\sH$ to $G-V(T)$ has a fractional packing;
 \item \label{item:TechInd4} Every correspondence-cover $(L_T,H_T)$ of $T$ with 
\begin{itemize}
 \item $|L_T(u)| = |L(u)|$, for every $u\in V(T)$ with no neighbour outside $T$, and
 \item $|L_T(u)| = |L(u)|-1$, for every $u\in V(T)$ with one neighbour $v$ outside $T$
\end{itemize}
has a fractional packing.
\end{enumerate}
The same holds mutatis mutandis for list-covers.
\end{lem}


\section{Bounds for Cartesian products and other layered decompositions}\label{sec:cartesianproducts}

In this section we prove~\cref{thm:treedepth} and~\cref{thm:cartesianproductFractPack}.
To do so, we leverage that the graphs involved admit a nice decomposition into layers that have lower fractional correspondence packing number and few edges between them. That property is encapsulated in the following definition.

\begin{defi}\label{def:chicbulletwidth}
The \emph{$\chi_c^{\bullet}$-width} of a graph $G$ is defined as the minimum integer $t\geq 1$ such that there exists a partition $V_1\cup V_2 \cup \ldots \cup V_m$ of the vertices of $G$ with the following properties:
\begin{itemize}
 \item $\chi_c^{\bullet}(G[V_i]) \leq t$ for all $1\leq i \leq m$, and
 \item every vertex in $V_i$ has at most one neighbour in $\bigcup_{j=1}^{i-1}V_j$, for all $2\leq i \leq m$.
\end{itemize}
\end{defi}

\begin{lem}\label{lem:chicwidth}
$\chi_{c}^{\bullet}(G)\leq t+1$ for every graph $G$ with $\chi_c^{\bullet}$-width $t$.
\end{lem}
\begin{proof}
Consider a $(t+1)$-fold correspondence-cover $\sH$ of $G$. Let $V_1\cup \ldots \cup V_m$ be a vertex partition of $G$ that certifies that $G$ has $\chi_c^{\bullet}$-width $t$. We proceed by induction on $m$. If $m=1$ then $G=G[V_m]$ has fractional correspondence packing number $\leq t$. If $m\geq 2$, choose the induced subgraph $T=G[V_m]$ and apply Lemma~\ref{lem:technicalfractionallemma} to find the desired fractional packing of $\sH$. This is allowed since (i) is satisfied by definition of the partition, (ii) holds because $t+1\geq 2$, (iii) holds because by induction $G-V(T)= G[\bigcup_{j=1}^{m-1} V_j]$ has fractional correspondence packing number $\leq t+1$, and (iv) holds because by assumption $\chi_c^{\bullet}(T)\leq t$.
\end{proof}

\begin{rem}
    From~\cref{lem:chicwidth}, it follows that the $\chi_c^{\bullet}$-width of a graph $G$ is either $\chi_c^{\bullet}(G)$ or $\chi_c^{\bullet}(G)-1$. 
    By replacing $\chi_c^{\bullet}$ with $\chi_c^{\star}$ in~\cref{def:chicbulletwidth} one could similarly define the \emph{$\chi_c^{\star}$-width} of a graph $G$. However, this integral analogue can be strictly below $\chi_c^{\star}(G)-1$, for instance when $G$ is a cycle. So far, the maximum possible difference between $\chi_c^{\star}$ and $\chi_c^{\star}$-width is unknown, even when restricting the definition to the case $m=2$ and $\abs{V_1}=1$, which corresponds with the change of $\chi_c^{\star}$ under addition of a universal vertex.
\end{rem}

As a first application of~\cref{lem:chicwidth}, we find that $\chi_c^{\bullet}(G)$ is bounded by the \emph{treedepth} of $G$, thus proving~\cref{thm:treedepth}. 

\begin{cor}\label{cor:treedepth}
For every graph $G$, $\chi_c^{\bullet}(G)$ is at most the \emph{treedepth} of $G$.
\end{cor}
\begin{proof}
We proceed by induction on the treedepth $d$ of $G$. We may assume $G$ is connected, so that some supergraph of $G$ has a spanning tree $D$ of height $d$ such that every edge of $G$ joins an ancestor-descendant pair of $D$. Let $v$ be the root vertex and choose the vertex partition $V_1=\{v\}$, $V_2=V(G)-\{v\}$. Since $G[V_2]$ has treedepth $d-1$, by induction it has fractional correspondence packing number $\leq d-1$. Moreover, every vertex in $V_2$ has at most one neighbour in $V_1$. So $G$ has $\chi_{c}^{\bullet}$-width at most $d-1$ and by Lemma~\ref{lem:chicwidth} $\chi_c^{\bullet}(G)\leq d$.
\end{proof}

Implicitly, \cref{lem:chicwidth} has already been applied in~\cite[Thm.~9.3]{CCZ23}, where it was proved that every planar graph $G$ of girth at least $6$ has $\chi_c^{\bullet}$-width at most two, and hence $\chi_c^{\bullet}(G)\leq 3$.\\

Next, we consider the Cartesian product of graphs. In~\cite{CCZ23} it was already observed that taking the Cartesian product with a tree increases $\chi_c^{\bullet}$ by at most one. For illustration purposes, we reprove it here with the terminology of $\chi_c^{\bullet}$-width.
\begin{cor}[\cite{CCZ23}]\label{cor:cartesianproductwithttree}
For every graph $G$ and every tree $F$, 
$$\chi_c^{\bullet}(G \square F) \leq \chi_c^{\bullet}(G)+1.$$
\end{cor}
\begin{proof}
Since $F$ is a tree, it admits a vertex ordering $v_1, \ldots, v_n$ such that every $v_i$ has at most one neighbour $v_j$ with $j<i$. Consider the vertex partition $V_1\cup \ldots \cup V_n$ of $G \square F$ given by $V_i:=V[G \square \{v_i\}]$, for all $i$. Then $G[V_i]$ is isomorphic to $G$ for all $i$, and each vertex in $V_i$ has at most one neighbour in $\bigcup_{j=1}^{i-1}V_j$. So $G\square F$ has $\chi_c^{\bullet}$-width at most $\chi_c^{\bullet}(G)$ and we can apply~\cref{lem:chicwidth}.
\end{proof}

\begin{cor}\label{cor:cartesin_for_n_trees}
For any $n$ trees $T_1,T_2,\ldots, T_n$, 
$$\chi_{\ell}(T_1\square T_2\square \ldots \square T_n)\leq \chi_c^{\bullet}(T_1\square T_2\square \ldots \square T_n)\leq n+1.$$
Moreover, these become equalities when $T_1=T_2=\ldots=T_n$ is a sufficiently large star.
\end{cor}

\begin{proof}
The upper bound follows from iteratively applying~\cref{cor:cartesianproductwithttree}. 
The sharpness is due to iteratively applying the more precise version of~\cref{thm:cartesianproductList},~\cite[Lem.~3]{BJKM06} where it was shown that $G_2$ (in~\cref{thm:cartesianproductList}) can be taken a large star when $d=1$. 
Taking $n$ times the largest star $T$ in the above procedure, we conclude for $T_1=T_2=\ldots=T_n=T$ that $\chi_{\ell}(T_1\square T_2\square \ldots \square T_n)=n+1.$
\end{proof}

For the Cartesian product of two arbitrary graphs $G_1,G_2$ we need to be more careful.

\begin{cor}\label{cor:combininglocaldegreewithcartesianproduct}
 For two graphs $G_1$ and $G_2$, let $\sH=(L,H)$ be a correspondence-cover of their Cartesian product $G_1\square G_2$, such that every vertex $(a,b) \in V(G_1 \square G_2)$ has a list $L(a,b)$ of size at least $\chi_c^{\bullet}(G_1) + \deg_{G_2}(b)$. Then $\sH$ has a fractional packing.
\end{cor}

\begin{proof}
For brevity we write $G:=G_1\square G_2$, with cover $\sH$. 
Here we may assume that $G_1$ and $G_2$ are connected.
By~\cref{lem:monotonicityfractionalpacking} we may assume that that each list $L(a,b)$ has size \emph{precisely} $\chi_c^{\bullet}(G_1) + \deg_{G_2}(b)$.
 We proceed by induction on the number of vertices $n$ of $G_2$. If $n=1$ then $G=G_1$ and each list has size $\chi_c^{\bullet}(G_1)$, so by definition $\sH$ has a fractional packing. So we may assume $n\geq 2$. Let $b_0$ be a vertex of $G_2$ of maximum degree, and consider the induced subgraph $T:= G_1 \square (G_2-\{b_0\})$ of $G$. An application of Lemma~\ref{lem:technicalfractionallemma} to $T$ will yield the desired fractional packing of $\sH$.\\
 
 So it suffices to check the conditions of Lemma~\ref{lem:technicalfractionallemma}, which we will do now. Condition (i) holds by definition of Cartesian product. Indeed, for a vertex $(a,b)$ in $V(T)$, its only potential neighbour outside $T$ is $(a,b_0)$. Condition (ii) holds because we chose $b_0$ of maximum degree in $G_2$ and because $\chi_c^{\bullet}(G_1)\geq 1$. Condition (iii) holds because $G-V(T)$ is isomorphic to $G_1$, so the restriction of $\sH$ to $G-V(T)$ (which has lists of size $\geq \chi_c^{\bullet}(G_1)$ for every vertex) has a fractional packing. Finally, condition (iv) holds by the induction hypothesis. Indeed, let $(L_T,H_T)$ be a cover of $T$ of the form prescribed in~\cref{lem:technicalfractionallemma}(iv). Let $(a, b)$ be a vertex of $T$. If $(a, b)$ has no neighbour outside $T$, then $b$ (as a vertex of $G_2$) does not have $b_0$ as a neighbour, so $$|L_T(a,b)|=|L(a,b)| \geq \chi_c^{\bullet}(G_1) + \deg_{G_2}(b)=\chi_c^{\bullet}(G_1) + \deg_{G_2-\{b_0\}}(b).$$
 On the other hand, if $(a,b)$ does have a neighbour outside $T$, then that neighbour must be $(a,b_0)$ (and $b_0$ is a neighbour of $b$ in $G_2$), so $$|L_T(a,b)|=|L(a,b)| -1 \geq \chi_c^{\bullet}(G_1) + \deg_{G_2}(b)-1 = \chi_c^{\bullet}(G_1) + \deg_{G_2-\{b_0\}}(b).$$
 As $T$ is the Cartesian product of $G_1$ and $G_2-\{b_0\}$, the latter having fewer than $n$ vertices, it follows by induction that $\sH_T$ has a fractional packing. This concludes the verification of the conditions (i)--(iv) of Lemma~\ref{lem:technicalfractionallemma}.
 \end{proof}

In particular~\cref{cor:combininglocaldegreewithcartesianproduct} implies that taking the Cartesian product with another graph can increase the fractional correspondence packing number by at most the maximum degree of that graph.
\begin{cor}\label{cor:cart_maxdegree}
 For every two graphs $G_1$ and $G_2$, 
 $$\chi_c^{\bullet}(G_1 \square G_2) \leq \chi_c^{\bullet}(G_1)+ \Delta(G_2).$$
\end{cor}
This, and verbatim repeating the proof of~\cref{cor:combininglocaldegreewithcartesianproduct} with list-covers instead of correspondence-covers, yields~\cref{thm:cartesianproductFractPack}.

\begin{rem}
Somewhat symmetric to~\cref{thm:cartesianproductFractPack}, a straightforward generalisation of the proof of~\cref{cor:treedepth} yields $\chi_c^{\bullet}(G_1 \square G_2) \leq \chi_c^{\bullet}(G_1)+ \mathrm{td}(G_2)-1$ for every two graphs $G_1$ and $G_2$. The induction is again on the treedepth of $G_2$, but now choose the vertex partition with $V_1=G_1 \square \{v\}$ and $V_2=G_1 \square (V(G_2)-\{v\})$.
\end{rem}

\section{Upper bound in terms of pathwidth}\label{sec:pw+1}
In this section we prove~\cref{thm:pathwidth}.
Recall that $\chi_c^{\bullet}(G) \leq k$ if and only if for every $k$-fold correspondence-cover $(L,H)$ of $G$ there exists a probability distribution on independent transversals $I$ of $(L,H)$ such that at each vertex $v$, each colour in $L(v)$ appears with equal probability $1/k$ in $I$. Every independent transversal $I$ can be identified with a unique \emph{correspondence-colouring} $c:V(G) \rightarrow V(H)$ given by $c(v):=I\cap L(v)$. In this section we will find such a probability distribution by inductively constructing an explicit multiset of correspondence-colourings $\C:=\{c_1,c_2,\ldots c_{k\cdot m}\}$ such that for each $v\in V(G)$ and each $x\in L(v)$, there are precisely $\frac{|\C|}{k}$ colourings $c\in \C$ with $c(v)=x$.\\

\begin{defi}
A $q$-fold correspondence-cover $(L,H)$ of a graph $G$ \emph{has full matchings} if for every $uv\in E(G)$, the matching between $L(u)$ and $L(v)$ is perfect, i.e., of size $q$. 
\end{defi}

As a warm-up, we first give the proof of the following proposition, which treats the case $\mathrm{pw}(G) \le 2.$

\begin{prop}\label{prop:pathwidth<=2}
 For every graph $G$ with $\mathrm{pw}(G)\le 2$, $\chi_c^{\bullet}(G)\le \mathrm{pw}(G)+1$.
\end{prop}
\begin{proof}
When $\mathrm{pw}(G)=1$, $G$ is a forest and the result immediately follows from the fact that $\chi_c^{\bullet}$ is at most twice the degeneracy~\cite[Thm.~9]{CCDK21}.
For graphs with pathwidth two, we verify the sharp inequality $\chi_c^{\bullet}(G)\le \mathrm{pw}(G)+1$ by adapting the argument from~\cite[Thm.~3.2]{BMS22}.
It suffices to treat maximal graphs with pathwidth two, so we can assume that $G$ is a $2$-caterpillar. Furthermore, we may assume that the given $3$-fold correspondence-cover of $G$ has full matchings. To prove the proposition, we will construct $6$ correspondence-colourings such that for a special subset of \emph{balanced} edges $uv$ of $G$, their restrictions to $G[u,v]$ consist of all $3\cdot 2$ correspondence-colourings that are allowed by the perfect matching between $L(u)$ and $L(v)$. We will make sure that in the end, every vertex $v$ will be incident to some balanced edge; note that this implies the desired property that every colour $x\in L(v)$ appears in precisely two of the six correspondence-colourings. 

The main idea is that the $6$ possible correspondence-colourings on a balanced edge $ab$ can be extended to a common neighbour $d$ of $a$ and $b$, in such a way that $ad$ becomes a balanced edge as well. Similarly, one can extend to $d$ so that $bd$ becomes balanced. (But in general it is not possible to make both $ad$ and $bd$ balanced!).

This is sketched in~\cref{fig:local_ext}, where we can assume without loss of generality that the $3$-fold cover of the triangle has the full identity matching on the edges $ab,ad$.
Up to isomorphism there are three possible perfect matchings for the edge $bd$.
We consider the extended colourings, $c_1$ for the identity matching (which forbids the colourings $(1,1),(2,2),(3,3)$ on $bd$), and $c_2$ for the matching that forbids the colourings $(1,2)$, $(2,1)$ and $(3,3)$ on $bd$, and $c_3$ for the perfect matching forbidding $c(d)\equiv c(b)+1 \pmod 3$.
In every case, we can extend by assigning a colour to $d$, such that also one fixed edge (either $bd$ or $ad$) becomes balanced, i.e., such that the $6$ possible combinations appear when restricting the colourings to this edge. For $c_3$, the extension is different depending on the desired balanced edge (in the table column for $c_3(d)$ in~\cref{fig:local_ext}, this extension is described on the left for $bd$, and on the right for $ad$).

Now one can iteratively use this type of extension along a maximum $2$-path of $G$
, and next extend towards the vertices with degree two. In this way, every vertex becomes incident to some balanced edge.
\end{proof}

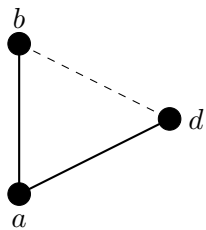
\begin{figure}[H]
\centering
 \begin{minipage}{\linewidth}
 \centering
 \begin{minipage}{0.35\linewidth}
 \centering
\begin{tikzpicture}
\foreach \x/\y in {0/0,2/1,0/2}{
\draw[fill] (\x,\y) circle (0.15);
}

\draw[dashed] (0,2)--(2,1);
\draw[thick] (2,1)--(0,0)--(0,2);

\node at (2.35,1){$d$};
\node at (0,2.35){$b$};
\node at (0,-0.35){$a$};

\end{tikzpicture} \end{minipage}
 \begin{minipage}{0.625\linewidth}
 \centering
\begin{tabular}{|c|c|c|c|c|}
 \hline
 c(a) & c(b) & $c_1(d)$ & $c_2(d)$ & $c_3(d)$\\
 \hline
 1 &2& 3 & 3 & 2 \\
 1 & 3&2& 2& 2 or 3\\
 2 &1& 3 & 3& 3 or 1 \\
 2 &3 &1&1 &3 \\
 3 & 1&2&1 &1 \\
 3 & 2&1&2 & 1 or 2 \\
 \hline
\end{tabular}
\end{minipage}
\end{minipage}

\caption{Local extension of six correspondence-colourings of a $2$-tree, from a balanced edge $ab$ to $d$, such that $ad$ or $bd$ becomes balanced as well. The bold edges indicate full identity matchings in the cover.
}\label{fig:local_ext}
\end{figure}

Before proving the general case, we state and prove the following lemma that contains the main ideas of extending the colourings.
The proof of~\cref{thr:chicbullet<pw+1} is then just the linear time algorithm for finding a fractional packing of any $(\mathrm{pw}(G)+1)$-fold correspondence-cover by applying the procedure of this lemma $O(n)$ many times.

\begin{lem}\label{lem:extensioncolouringswithin_Kp+1}
 Let $G \cong K_{p+1}$ be a complete graph with vertices $v_1,v_2, \ldots, v_p,v_{p+1}$. Let $(L,H)$ be a $(p+1)$-fold correspondence-cover of $G$ with full matchings.
 Assume there exists a set $\C$ of $(p+1)!$ distinct correspondence-colourings of $G-\{v_{p+1}\}$ such that for every $1\leq i \leq p$, the restrictions of $\C$ to $G[v_i,v_{i+1},\ldots, v_{p}]$ form a multiset $\C_{[i,p]}$ of $\frac{(p+1)!}{i!}$ distinct correspondence-colourings, each appearing with multiplicity $i!$.
 Then we can extend each colouring in $\C$ to $v_{p+1}$, in such a way that we obtain a set $\D_{[1,p+1]}$ of $(p+1)!$ distinct correspondence-colourings of $G$, such that for every $1\leq i \leq p+1$ the restrictions of $\D_{[1,p+1]}$ to $G[v_i,v_{i+1},\ldots, v_{p+1}]$ form a multiset $\D_{[i,p+1]}$ of $\frac{(p+1)!}{(i-1)!}$ distinct correspondence-colourings, each appearing with multiplicity $(i-1)!$.
\end{lem}

\begin{proof}
For $1\leq i \leq j \leq p+1$, let $G_{[i,j]}$ denote the induced subgraph $ G[v_i,v_{i+1},\ldots, v_{j}]$. We prove the following statement by induction on $i$, starting from $i=p+1$ descending to $i=1$:\\
There exists a multiset $\D_{[i,p+1]}$ of correspondence-colourings of $G_{[i,p+1]}$ such that
\begin{enumerate}[(i)]
    \item $\D_{[i,p+1]}$ is a multiset of $\frac{(p+1)!}{(i-1)!}$ distinct correspondence-colourings of $G_{[i,p+1]}$, each appearing with multiplicity $(i-1)!$, and
    \item 
    The restriction of $\D_{[i,p+1]}$ to $G_{[i+1,p+1]}$ is equal to $\D_{[i+1,p+1]}$, and
    \item The restriction of $\D_{[i,p+1]}$ to $G_{[i,p]}$ is  equal to $\C_{[i,p]}$.
\end{enumerate}
This implies the lemma, because the first and second part of the induction hypothesis ensure that the multisets $\D_{[i,p+1]}$ have the desired multiplicities, and the third part of the induction hypothesis with $i=1$ ensures that $\D_{[1,p+1]}$ indeed consists of extensions to $v_{p+1}$ of the colourings in $\C_{[1,p]}$. Note that the second part in fact implies that for all $j \in [i,p+1]$, the restriction of $\D_{[i,p+1]}$ to $G_{[j,p+1]}$ is equal to $\D_{[j,p+1]}$. \\

For the base case $i=p+1$, we can simply choose $\D_{[p+1,p+1]}$ to be the multiset containing each element of $L(v_{p+1})$ with multiplicity $p!$. These are colourings of $G_{[p+1,p+1]}=G[v_{p+1}]$.
For clarity we will treat the next case $i=p$ separately as well, even though it also follows from the induction argument below. There are precisely $(p+1)\cdot p$ colourings of $G_{[p,p+1]}$ that avoid the size $p+1$ matching between $L(v_p)$ and $L(v_{p+1})$ in the cover graph, and we let $\D_{[p,p+1]}$ be the multiset that contains each of them with multiplicity $(p-1)!$. Clearly the restriction of $\D_{[p,p+1]}$ to $G_{[p+1,p+1]}$ is equal to $\D_{[p+1,p+1]}$. And symmetrically, the restriction of $\D_{[p,p+1]}$ to $G_{[p,p]}=G[v_p]$ is equal to $\C_{[p,p]}$ as it contains every colour in $L(v_p)$ with multiplicity $p!$.\\

From now on we assume the induction hypothesis is true for all integers larger than $i$ and at most $p+1$. We want to prove it for $i$. For the induction step, we will construct $\D_{[i,p+1]}$ from $\C_{[i,p]}$ and $\D_{[i+1,p+1]}$.

By the conditions of the lemma, $Z:=\C_{[i+1,p]}$ consists of $(p+1)!/(i+1)!$ colourings of multiplicity $(i+1)!$, and its multiset of extensions $\C_{[i,p]}$ to $v_i$ consists of $(p+1)!/i!$ colourings of multiplicity $i!$. This implies that for every $\vec{c}:=(c(v_{i+1}), \ldots, c(v_p)) \in Z $ there is a set $S_{i}(\vec{c})$ of precisely $i+1$ possible colours $c(v_i)$ for $v_i$ such that $(c(v_i),c(v_{i+1}), \ldots, c(v_p)) \in \C_{[i,p]} $.

On the other hand, let $\D_{[i+1,p]}$ denote the multiset of restrictions of $\D_{[i+1,p+1]}$ to $G_{[i+1,p]}$. By the third part of the induction hypothesis, we have that $\D_{[i+1,p]}=Z$.

By the first part of the induction hypothesis (applied with $i+1$), we have that $\D_{[i+1,p+1]}$ consists of $(p+1)!/i!$ colourings of multiplicity $i!$. Similar to above we conclude that for every $\vec{c}=(c(v_{i+1}), \ldots, c(v_p)) \in Z $ there is a set $S_{p+1}(\vec{c})$ of precisely $i+1$ possible colours $c(v_{p+1})$ for $v_{p+1}$ such that $(c(v_{i+1}), \ldots, c(v_{p+1})) \in \D_{[i+1,p+1]} $.\\

To construct $\D_{[i,p+1]}$ we now want to extend each $\vec{c} \in Z$ to $\{v_i, v_{p+1}\}$ simultaneously, by using colour pairs from $S_i(\vec{c}) \times S_{p+1}(\vec{c})$ in some balanced manner.
One way to do that would be to select every pair equally often, but that is not allowed in general because in the cover graph, the matching between $L(v_i)$ and $L(v_{p+1})$ may forbid some of those pairs.
Let $M(\vec{c})$ be the submatching between $L(v_i) \cap S_i(\vec{c})$ and $L(v_{p+1}) \cap S_{p+1}(\vec{c})$ and augment it to an arbitrary auxiliary perfect matching $\overline{M}(\vec{c}) \supseteq M(\vec{c})$ of a complete bipartite graph with parts $S_i(\vec{c})$ and $S_{p+1}(\vec{c})$. Note that $\overline{M}(\vec{c})$ has size $|S_i(\vec{c})|=i+1$. This leaves $i(i+1)$ pairs for $(c(v_i), c(v_{p+1})) \in S_i(\vec{c}) \times S_{p+1}(\vec{c})$ which avoid $\overline{M}(\vec{c})$ and which we may use to extend $\vec{c}$ to $\{v_i, v_{p+1}\}$. Taking each of these $i(i+1)$ extensions with multiplicity $(i-1)!$, and this for every \emph{distinct} $\vec{c} \in Z$, yields our desired construction of $\D_{[i,p+1]}$. 
We stress that the matching $\overline{M}(\vec{c})$ purely serves as an auxiliary tool to extend one particular colouring $\vec{c}$ \emph{in a balanced manner} ensuring that each colour in $S_i(\vec{c}) \cup S_{p+1}(\vec{c})$ is chosen equally often; it does not impose any restrictions on which colours are allowed for extension in other steps of this proof.

The first part of the induction hypothesis holds because each of the $(p+1)!/(i+1)!$ distinct colourings $\vec{c} \in Z$ has been mapped to $i(i+1)$ distinct colourings of $G_{[i,p+1]}$ that only differ from each other on $\{v_i,v_{p+1}\}$, thus yielding $(p+1)!/(i-1)!$ distinct colourings of $G_{[i,p+1]}$.

The second part of the induction hypothesis is true because for every distinct $(\vec{c}, c(v_{p+1}))\in \D_{[i+1,p+1]}$ we extended $\vec{c}\in \C_{[i+1,p]}$ to $\{v_i,v_{p+1}\}$ in a balanced manner, by choosing colour $c(v_{p+1}) \in S_{p+1}(\vec{c})$ for $v_{p+1}$ with multiplicity $(i-1)!$ for each of the $i$ allowed choices of $c(v_{i})\in S_{i}(\vec{c})$. This ensured that $(\vec{c}, c(v_{p+1}))$ appears with multiplicity $i\cdot (i-1)! = i!$ in the restriction of $\D_{[i,p+1]}$ to $G_{[i+1,p+1]}$, as is required for $\D_{[i+1,p+1]}$.

The third part of the induction hypothesis is true for essentially the same reason, but we will write it out for completeness. For every distinct $(c(v_i),\vec{c})\in \C_{[i,p]}$ we extended $\vec{c}\in \C_{[i+1,p]}$ to $\{v_i,v_{p+1}\}$ in a balanced manner, by choosing colour $c(v_i) \in S_i(\vec{c})$ for $v_i$ with multiplicity $(i-1)!$ for each of the $i$ allowed choices of $c(v_{p+1})\in S_{p+1}(\vec{c})$. This ensured that $(c(v_i),\vec{c})$ appears with multiplicity $i\cdot (i-1)! = i!$ in the restriction of $\D_{[i,p+1]}$ to $G_{[i,p]}$, as is required for $\C_{[i,p]}$.
\end{proof}

\begin{thr}\label{thr:chicbullet<pw+1}
 For every graph $G$, $\chi_c^{\bullet}(G)\le \mathrm{pw}(G)+1$.
\end{thr}

\begin{proof}
Let $p:=\mathrm{pw}(G)$. It is sufficient to consider $p$-caterpillars $G$ (by~\cite{proskurowski1989maximal}), whose $(p+1)$-fold correspondence-covers have full matchings. We will prove that there are $(p+1)!$ proper correspondence-colourings, such that every vertex is coloured with each of its possible colours $p!$ many times.
First we set up some notation. Let $n$ be the number of vertices of $G$ and let $A_1,A_2,\ldots, A_{n-p}$ be a collection of active vertex sets that constructs the $p$-caterpillar (see~\cref{sec:defnot} for the definition of active sets), and recall that they are each of size $p+1$.
Let $O_f:=v_1v_2\ldots v_n$ be the order in which the vertices are added, where 
$A_1=\{v_1,\ldots,v_{p+1}\}$ and 
$v_{p+i}=A_{i}\setminus A_{i-1}$ for all $i>1$. The graph can also be constructed by going through the active sets in the reverse order
$A_{n-p},\ldots, A_2,A_1$. This yields another vertex ordering $O_r$, starting with the vertices of $A_{n-p}$
, followed by the vertices $A_{n-p-1}\setminus A_{n-p}, \ldots, A_{2}\setminus A_3, A_1\setminus A_2$.
We also define the opposite ordering of $O_r$, which is
$\overrightarrow{O_r}=A_1\setminus A_2, A_2\setminus A_3,\ldots, A_{n-p-1} \setminus A_{n-p}, \overrightarrow{A_{n-p}}$. Here $\overrightarrow{A_{n-p}}$ denotes $A_{n-p}$ equipped with the order that is opposite of the order in $O_r$.
More generally, we write $\overrightarrow{S}$ for a vertex subset $S$ equipped with the ordering inherited from $\overrightarrow{O_r}$. 
Note that $O_f$ can be very different from $\overrightarrow{O_r}$ (and also different from $O_r$).
Whereas $O_f$ is the order in which the vertices enter the active set,
$\overrightarrow{O_r}$ is the order in which they leave the active set. 
The $p$-caterpillar could be such that some vertices leave the active set in the iteration immediately after they entered it (these are the $p$-leaves), while
other vertices stay in the active set for a large number of iterations before they finally leave.
The key property of $\overrightarrow{O_r}$ that we will
use is that $A_i\setminus A_{i+1}$ is the lowest element in $\overrightarrow{A_{i}}$, for all $i$.

Suppose we are given a subgraph $\Gamma$ of $G$, an ordered subset $\overline{S}:=s_1,s_2,\ldots, s_{|S|}$ of at most $p+1$ vertices of $\Gamma$, and a collection $\C$ of $(p+1)!$ distinct correspondence-colourings of $\Gamma$. 
Then we say that $S$ is \emph{balanced} if the restriction of $\C$ to $\Gamma[S]$ is a multiset consisting of $\frac{(p+1)!}{(p+1-|S|)!}$ colourings, each appearing with multiplicity $(p+1-|S|)!$. Moreover, $\overline{S}$ is \emph{fully balanced} if $\{s_j,s_{j+1},\ldots, s_{|S|}\}$ is balanced for all $1\leq j \leq |S|$.\\

We now prove by induction on $i$ that there are $(p+1)!$ correspondence-colourings of $G_i:=G[\bigcup_{j=1}^{i} A_j]$, such that $\overrightarrow{A_j}$ is fully balanced for every $1\leq j \leq i$. As $G=G_{n-p}$ this will prove the theorem. We stress that we are performing induction in the order governed by $O_f$, but that the balancedness of the subsets $A_i$ will be derived with respect to the order $\overrightarrow{O_r}$.\\

To derive the initial condition for $i=1$, the proof of~\cref{lem:extensioncolouringswithin_Kp+1} can be applied $p$ times. 
We have $A_1=\{v_1,\ldots, v_{p+1}\}$ and we write $\overrightarrow{A_1}=\{v_{j_1},\ldots, v_{j_{p+1}}\}$. Start with the $p+1$ possible colours of $v_{j_1}$, each appearing with multiplicity $p!$, so that $\{v_{j_1}\}$ becomes (fully) balanced. One can then successively extend to $G[v_{j_1},v_{j_2}], G[v_{j_1},v_{j_2},v_{j_3}],\ldots, G[A_1]$, ultimately making $\overrightarrow{A_1}$ fully balanced in $G_1$.
(If the reader is uncomfortable with this argument via the proof of~\cref{lem:extensioncolouringswithin_Kp+1}, here is an alternative way to obtain the initial condition. First extend $G$ to a larger $p$-caterpillar $G^{*}$ by appending vertices $v_{-p},\ldots,v_{-1},v_{0}$ at the start, and choosing a full identity matching in the cover for every new edge $e\in E(G^*)\setminus E(G)$. Because the new matchings are full identity, the complete subgraph  $G^*[v_{-p}\ldots, v_{-1},v_0]$  clearly has precisely $(p+1)!$ correspondence-colourings with respect to which $\overrightarrow{\{v_{-p}, \ldots, v_{-1},v_0 \}}$ is fully balanced. Then we may relabel $v_{i}$ to $v_{i+p+1}$ for all $i\geq -p$ and apply the argument below to $G^*$ instead of $G$, which is sufficient since $G$ is a subgraph of $G^*$.)\\

Next, assume that the induction hypothesis holds for $G_i$, for some $1\leq i \leq n-p-1$. We prove it for $i+1$. 
Since $\overrightarrow{A_i}$ is fully balanced and (crucially) since $A_i\setminus A_{i+1}$ is the lowest element in the order of $\overrightarrow{A_i}$, it follows that $\overrightarrow{A_i \cap A_{i+1}}$ is fully balanced as well.
Therefore~\cref{lem:extensioncolouringswithin_Kp+1} can be applied to the ordered clique $G_i[\overrightarrow{A_i \cap A_{i+1}}]$, which yields that the correspondence-colourings of $G_i$ can be extended to $G_{i+1}$ by colouring $v_{p+i+1}=A_{i+1}\setminus A_i$ such that $(\overrightarrow{A_i \cap A_{i+1}},v_{p+i+1})$
becomes fully balanced. Note that this is $A_{i+1}$ but equipped with a different ordering than we need.
So from this we need to derive that $\overrightarrow{A_{i+1}}$ is fully balanced as well.

Write $\overrightarrow{A_{i} \cap A_{i+1}}:=a_1,a_2,\ldots,a_{p}$ and $\overrightarrow{A_{i+1}}:=a_1,a_2,\ldots, a_{j}, v_{p+i+1}, a_{j+1},\ldots,a_{p}$.
Then for every $r\geq j+1$, we have that $\{a_{r}, a_{r+1},\ldots, a_{p}\}$ is balanced because $\overrightarrow{A_i \cap A_{i+1}}$ is fully balanced. Moreover, for every $1 \leq r \leq j$ we have that $\{a_r,\ldots a_j,v_{p+i+1},a_{j+1},\ldots, a_{p}\}=\{a_r,\ldots, a_{p}, v_{p+i+1}\}$ is balanced because $(\overrightarrow{A_i \cap A_{i+1}},v_{p+i+1})$ is fully balanced. For the same reason, $\{v_{p+i+1},a_{j+1},\ldots, a_{p}\}$ is balanced. By definition, it follows that $\overrightarrow{A_{i+1}}$ is fully balanced.
None of the previous active sets $\overrightarrow{A_j}$ (for $j\leq i$) contains $v_{p+i+1}$, so they remain fully balanced after the extension of the colourings to $v_{p+i+1}$. So the statement holds for $G_{i+1}$, as desired.
\end{proof}

\begin{examp}
An example of a graph with pathwidth three is given in~\cref{fig:constr_pw3}, where $O_f=v_1v_2v_3v_4v_5v_6v_7uv_8v_9$ and $O_r=v_9v_8v_7v_6v_4uv_5v_3v_2v_1$ (note that $u$ is the only pendant vertex not on the $3$-path).
In this case,~\cref{lem:extensioncolouringswithin_Kp+1} is applied consecutively to $$G[v_1v_2v_3v_4], G[v_2v_3v_5v_4], G[v_3v_5v_4v_6], G[v_5v_4v_6v_7], G[uv_4v_6v_7], G[v_4v_6v_7v_8], G[v_6v_7v_8v_9].$$
We sketch a situation for one such step.
Assume there are identity matchings on $v_1v_2, v_1v_3, v_2v_3$ and $v_3v_4$,
and the matching between $v_2$ and $v_4$ implies $c(v_4) \not\equiv c(v_2)+1 \pmod 4.$
Then the $24$ colourings for $G[v_1v_2v_3]$ are easy to determine (each of the vertices $v_1, v_2$ and $v_3$ gets assigned a different colour from $[4]$).
In~\cref{tab:initialisation}, we present how to extend these colourings to $v_4.$
Given $c(v_3)$, there are three possible colours for $v_4$ (which will appear each twice for that choice of $c(v_3)$).
Among the choices where $c(v_3)=1$, we consider the additional colour for $c(v_2)$ and list the choices for $c(v_4)$. To aim for a balanced situation, we need to forbid the colour $2$ ourselves in the last case.
Since $c(v_1)$ is known as well, we can consider to match them, which is always possible.
In the case where we have the identity matching between $L(v_1)$ and $L(v_4)$ as well, the colourings $(c(v_1),c(v_2),c(v_3),c(v_4))$ for which $c(v_3)=1$ will be 
$$(3,2,1,4),(4,2,1,2),(2,3,1,3),(4,3,1,2),(2,4,1,3),(3,4,1,4).$$
The other $18$ colourings are determined analogously,
and also the extension towards the next vertices $v_i$ works the same.

\begin{table}[h]
 \centering
\begin{tabular}{|c|c|}
 \hline
 $c(v_3)$ & $c(v_4)$ \\
 \hline
 1 & \{2,3,4\}\\
 2& \{1,3,4\}\\
 3& \{1,2,4\}\\
 4& \{1,2,3\}\\
 \hline
\end{tabular}\quad
\begin{tabular}{|c|c|c|}
 \hline
 $(c(v_3),c(v_2))$ & $c(v_4)$ & $c(v_1)$ \\
 \hline
 (1,2) & \{2,4\}& \{3,4\}\\
 (1,3)& \{2,3\}& \{2,4\}\\
 (1,4)& \{\textcolor{red}{2,}3,4\}& \{2,3\}\\
 \hline
\end{tabular}
\caption{Determining the colourings iteratively}
\label{tab:initialisation}
\end{table}

\begin{figure}[ht]
 \centering
 \begin{tikzpicture}[scale=1.5]
 \definecolor{cv0}{rgb}{0.0,0.0,0.0}
 \definecolor{c}{rgb}{1.0,1.0,1.0}

 \Vertex[L=\hbox{$v_1$},x=-2,y=0]{v0}
 
 \Vertex[L=\hbox{$v_2$},x=-1.73,y=1]{v1}
 \Vertex[L=\hbox{$v_3$},x=-1,y=1.73]{v2}
 
 \Vertex[L=\hbox{$v_5$},x=0,y=2]{v3}
 
 \Vertex[L=\hbox{$v_4$},x=0,y=0]{v4}

 \Vertex[L=\hbox{$v_6$},x=1,y=1.73]{v5}

 \Vertex[L=\hbox{$v_8$},x=3,y=1.73]{v7}
 \Vertex[L=\hbox{$v_7$},x=2.,y=0]{v6}
 \Vertex[L=\hbox{$v_9$},x=4,y=0]{v8}

 \Vertex[L=\hbox{$u$},y=0.3,x=1.2]{u}

 \draw[line width=0.65mm] (v2)--(v4)--(v0)--(v1)--(v2)--(v3)--(v5)--(v7)--(v8)--(v6)--(v4)--(v1);
 \draw[line width=0.65mm] (v3)--(v4)--(v5)--(v6)--(v7);
\draw[line width=0.65mm,color=blue,style=dashed] (v1)--(v3)--(v6);
\draw[line width=0.65mm,color=blue,style=dashed] (v0)--(v2)--(v5)--(v8);
\draw[line width=0.65mm,color=blue,style=dashed] (v4)--(v7);
\draw[line width=0.5mm,color=red,style=dotted] (v4)--(u)--(v5);
\draw[line width=0.5mm,color=red,style=dotted] (v6)--(u);
 \end{tikzpicture}
 \caption{A graph with pathwidth three.\label{fig:constr_pw3}}
\end{figure}
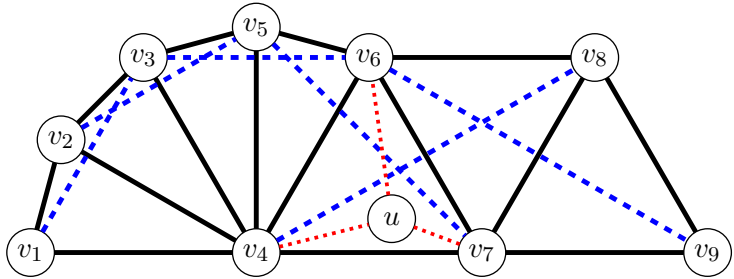
\end{examp}

\begin{rem}

 A key feature of our strategy in~\cref{lem:extensioncolouringswithin_Kp+1} and its application in~\cref{thr:chicbullet<pw+1} is that
we make the correspondence-colourings balanced ``in only one direction'', based on the order in which the $p$-caterpillar is constructed.
It is tempting to think that at the expense of increasing the list sizes we might ignore this order,
thus enabling the construction of balanced colourings on an arbitrary $p$-tree.
This would allow an upper bound on $\chi_c^{\bullet}$ in terms of treewidth. 
Surprisingly, this strategy fails spectacularly.
More specifically,~\cref{lem:extensioncolouringswithin_Kp+1} yields balanced colourings on any subset of the form $\{v_i,v_{i+1},\ldots,v_{p+1}\}$ but it becomes false if
instead we want balanced colourings on \emph{every} subset of $\{v_1,\ldots,v_{p+1}\}$.
It does not even work for $p=2$. Indeed, for any $q\geq 3$ there is a $q$-fold cover $\sH$ of the triangle $v_1v_2v_3$ that does not admit a collection of $q(q-1)(q-2)$ correspondence-colourings such that every edge is coloured with $q(q-1)$ different correspondence-colourings, each with multiplicity $q-2.$ 
For $\sH$, one can take the identity matchings on $v_1v_2$ and $v_1v_3$,
and the cyclically shifted matching on $v_2v_3$ consisting of edges $1_{v_2}2_{v_3}, 2_{v_2}3_{v_3}, \ldots, (q-1)_{v_2}q_{v_3}, q_{v_2}1_{v_3}$.
\end{rem}

\section{A flexibility bound in terms of degeneracy}\label{sec:degeneracy}
In this section we prove~\cref{thm:wflex_ddegenerate}. Recall that a graph $G$ is \emph{weighted $\epsilon$-flexible} with respect to a list-assignment $L$ if there exists a probability distribution on $L$-colourings $c$ of $G$, such that for every vertex $v\in V(G)$ and every colour $x\in L(v)$, $\mathbb{P}(c(v)=x) \geq \epsilon$. 
Dvo\v{r}\'{a}k, Norin and Postle~\cite{DNP19} proved that for every fixed $d$ there exists an (extremely small) $\epsilon_d$ such that every $d$-degenerate graph is weighted $\epsilon_d$-flexible with respect to every $(d+2)$-fold list-assignment.
Later, Kaul, Mathew, Mudrock and Pelsmajer~\cite{KMMP22} proved that the graph is (nonweighted!) $\epsilon_d$-flexible with the much better value $\epsilon_d= (\frac{1}{2})^{d+1}$. In this section we derive the same for \emph{weighted} $\epsilon_d$-flexibility. We do so via a local argument, allowing different list sizes for different vertices. Moreover, the result holds in the context of correspondence colouring. First we need two definitions.

\begin{defi}
A correspondence-cover $\sH=(L,H)$ of a graph $G$ has \emph{profile} $\mathcal{S}:=(s_v)$ if $|L(v)|=s_v$ for every vertex $v\in V(G)$.
\end{defi}
\begin{defi}
Given are a correspondence-cover $\sH$ of a graph $G$ and a collection $\mathcal{E}:=(\epsilon_v)$ of numbers $\epsilon_v\in [0,1]$ associated with the vertices $v\in V(G)$. Then $G$ is \emph{weighted $\mathcal{E}$-flexible with respect to $\sH$} if: there exists a probability distribution on independent transversals $I$ of $\sH$ such that $\mathbb{P}(x\in I)\geq \epsilon_v$ for all $v\in V(G)$. Moreover, $G$ is \emph{weighted $\mathcal{E}$-flexible} with respect to profile $\mathcal{S}$ if it is weighted $\mathcal{E}$-flexible with respect to every correspondence-cover with profile $S$.
\end{defi}
In terms of these two definitions, we can restate the following result from~\cite[Lem.~20]{CCDK23}.

\begin{lem}[\cite{CCDK23}]\label{lem:localdegreebound} 
Every graph $G$ is weighted $(\frac{1}{\deg(v)+1})-$flexible with respect to the profile $(\deg(v)+1)$.
\end{lem}

This implies in particular that $\chi_c^{\bullet}(G)$ is bounded from above by the maximum degree plus one. In order to obtain a similar bound in terms of degeneracy, one would be tempted to replace the degree $\deg(v)$ with oriented degree $\deg^+(v)$, with respect to some degeneracy ordering. However, that modified statement is false because by~\cite[Prop.~21]{CCDK23} there exist $d$-degenerate graphs with $\chi_c^{\bullet}(G)\geq d+2$. It is not known whether $\chi_c^{\bullet}(G)\leq d+2$. Does~\cref{lem:localdegreebound} remain true under replacement of $\deg(v)$ with $\deg^+(v)+1$? As a step in that direction, here is our local result that allows for a proof by induction.

\begin{lem}\label{lem:degeneracy_weightedflexibilty}
Let $G$ be a graph equipped with some vertex ordering $v_1,v_2,\ldots,v_n$. Let $\deg^+(v_i)$ denote the number of neighbours $v_j$ of $v_i$ with $j>i$. Then $G$ is weighted $(2^{-(\deg^+(v)+1)})$-flexible with respect to the profile $(\deg^+(v)+2)$. 
\end{lem}
\begin{proof}
We proceed by induction on $n$, the case $n=1$ being clear because in that case we can take a uniformly random colour from the size-$2$ list of $v_1$.
To prepare for the induction step, let $\sH=(L,H)$ be a correspondence-cover of $G$ with profile $(\deg^+(v)+2)_{v\in V(G)}$.
By adding edges to $H$ if necessary, we may assume that for every $ab\in E(G)$ the matching between $L(a)$ and $L(b)$ is maximum, i.e., of size $\min(|L(a)|,|L(b)|)$.

Since $v_n$ has no neighbours with higher index, we have $|L(v_n)|=2$. Given any $x\in L(v_n)$, let $\mathcal{H}_x:=(L_x,H_x)$ denote the correspondence-cover obtained by deleting $L(v_n)$ and the neighbours of $x$ from $H$.
By induction there exists a probability distribution $\mathbb{P}_x$ on independent transversals $I_x$ of $\sH_x$ such that $\mathbb{P}_x( y \in I_x) \geq (1/2)^{\deg_{G-v_n}^+(u)+1}$ for all $u\in V(G)-v_n$ and $y \in L_x(u)$.
Now fix $u$ a vertex in $G-v_n$ and $y \in L(u)$. We distinguish three cases.
\begin{enumerate}[(i)]
 \item If $u\notin N(v_n)$. Then $L_x(u)=L(u)$ and so $\deg_{G-v_n}^+(u)=\deg^+(u)$, so $\mathbb{P}_x(y \in I_x) \geq (1/2)^{\deg^+(u)+1}$.

\item If $u \in N(v_n)$ and $y \in N_H(x)$. Then $y$ is not in $H_x$, so $\mathbb{P}_x(y \in I_x)=0$. 
\item If $u \in N(v_n)$ but not $y \in N_H(x)$. Then $y\in H_x$ and $\deg_{G-v_n}^+(u)= \deg^+(u)-1$. Therefore $|L_x(u)|=\deg^+(u)+1$ and so $\mathbb{P}_x(y \in I_x) \geq (1/2)^{\deg^+(u)}$.
\end{enumerate}

We now select a uniformly random element $x$ of $L(v_n)$ and conditioned on that choice select (according to $\mathbb P_x$) a random independent transversal $I_x$ of $\sH_x$. The union $I$ of $x$ and $I_x$ is our random independent transversal of $\sH$ that certifies the statement of the lemma. 

Indeed, by construction we have for every $y^{\star}\in L(v_n)$ that $\mathbb{P}(y^{\star} \in I)=\frac{1}{|L(v_n)|}= \frac{1}{2}=(1/2)^{\deg^+(v_n)+1}$, as desired. And for $u$ a vertex in $G-v_n$ and $y \in L(u)$
we have $\mathbb{P}(y \in I)= \sum_{x \in L(v_n)} \mathbb{P}_x(y \in I_x) \cdot \frac{1}{|L(v_n)|}$. 
If $u\notin N(v_n)$, we are in the first case so we obtain $\mathbb{P}(y \in I) \geq (1/2)^{\deg^+(u)+1}$. If $u\in N(v_n)$ and $y \notin N_H(L(v_n))$, every choice of $x$ leads to the third case, so we obtain $\mathbb{P}(y \in I) \geq (1/2)^{\deg^+(u)}$. Finally, if $y \in N_H(L(v_n))$ then with probability $\frac{1}{|L(v_n)|}=\frac{1}{2}$ we have $I_x$ of the form in the second case and with the remaining probability half we are in the third case. So, together we have $\mathbb{P}(y \in I) \geq \frac{1}{2} \cdot 0 + \frac{1}{2} \cdot (1/2)^{\deg^+(u)}= (1/2)^{\deg^+(u)+1}$.
Thus for every choice of $u\in V(G)$ and $y \in L(u)$ the required inequality $\mathbb{P}(y \in I) \geq (1/2)^{\deg^+(u)+1}$ has been verified.
\end{proof}

\begin{cor}
Every $d$-degenerate graph $G$ is weighted $(1/2)^{d+1}$-flexible with respect to every $(d+2)$-fold correspondence-cover.
\end{cor}
\begin{proof}
Choose a vertex ordering $v_1, v_2,\ldots v_n$ such that every vertex $v_i$ has at most $d$ neighbours with higher index $j>i$, 
and apply Lemma~\ref{lem:degeneracy_weightedflexibilty} to the profile $(\deg^+(v)+2)_{v\in V(G)}$.
Since $\frac{i+2}{d+2}2^{-(i+1)} \ge 2^{-(d+1)}$ for every $\deg^+(v)=i \le d$, we conclude.
\end{proof}

The fact that planar graphs are $5$-degenerate yields the next corollary.
\begin{cor}
Every planar graph $G$ is weighted $\frac{1}{64}$-flexible with respect to every $7$-fold correspondence-cover.
\end{cor}

It remains an open problem whether this can be improved further. Are planar graphs weighted $\frac{1}{7}$-flexible with respect to every $7$-fold correspondence-cover? By definition, this would mean that planar graphs have fractional correspondence packing number at most $7$, beating the currently best bound $8$ that was obtained independently by~\cite{CCZ23} and~\cite{CS-R24+}.

\section{Concluding questions}\label{sec:conclusion}

The $n$-dimensional hypercube $Q_n$ is the iterated Cartesian product of $n$ copies of $K_2$. ~\cref{thm:cartesianproductList} and~\cref{cor:cartesianproductwithttree} both imply that $\chi_c^{\bullet}(Q_n)\leq n+1$ . For $n\leq 3$, there is equality $\chi_c^{\bullet}(Q_n)= n+1$; for $n\leq 2$ this is easy to see and for $Q_3$ it follows from the construction in~\cref{fig:hypercube_cover} in~\cref{sec:appendixQ3}. However, since $Q_n$ is bipartite and its vertices have degree $n$, asymptotically $\chi_c^{\bullet}(Q_n)$ must go to $\Theta(n/ \log n)$, by~\cite[Thm.~11]{CCDK21} and~\cite{Ber16}. It is thus clear that for large values of $n$, the current layering arguments are not powerful enough. We propose $Q_n$ as a testing ground to develop more sophisticated techniques.\\ 

The same can be asked for the iterated Cartesian product of $n$ paths or even cycles, for which again the asymptotic value of $\chi_c^{\bullet}$ is $ \Theta(n/ \log n)$.
Kaul et al.~\cite[Prop. 18]{KMMP22} proved that $\chi_{\ell}^{\bullet}(P_a \square P_b)=3$ for the Cartesian product of any two paths $P_a$ and $P_b$. With~\cref{cor:cartesianproductwithttree}, we find that the same equality $\chi_{\ell}^{\bullet}=\chi_{c}^{\bullet}=3$ holds for the Cartesian product of any two nontrivial trees. Similarly, since $\chi_c^{\bullet}(Q_3)=4$, we have $\chi_{c}^{\bullet}=4$ for the Cartesian product of any three nontrivial trees. Because $\chi_c^{\bullet}(Q_n)$ grows sublinearly, this pattern cannot continue: for large enough $n$ it ceases to be true that $\chi_{c}^{\bullet}=n+1$ for the Cartesian product $T_1 \square T_2 \square\ldots \square T_n$ of every $n$ nontrivial trees $T_1,T_2\ldots, T_n$. However,~\cref{cor:cartesin_for_n_trees} shows that $n+1$ is still a sharp upper bound, since it is attained by \emph{some} collection of $n$ trees. In the proof, we tacitly used that $\chi_{\ell}$ and $\chi_c^{\bullet}$ are equal for trees, but for graphs other than trees, we cannot leverage that trick anymore. Therefore we pose the $\chi_c^{\bullet}$-analogue of~\cref{thm:cartesianproductList} as a question.

 \begin{q}\label{q:cart_construction}
Given an arbitrary graph $A$ and integer $b\geq 1$, does there exist a graph $B$ with $\chi_c^{\bullet}(B)=b$ such that $\chi_c^{\bullet} (A \square B) \geq \chi_c^{\bullet}(A) + b - 1$ ?
 \end{q}

In \cite[Thm.~4]{KMSS23} this was confirmed for $\chi_c$ instead of $\chi_c^{\bullet}$, and for $B$ a complete bipartite graph $K_{b-1,t}$ with $t$ sufficiently large. By~\cite[Prop. 22]{CCDK23}, we have $\chi_{c}^{\bullet}(K_{b-1,t})=b$. These two facts together imply that the answer to~\cref{q:cart_construction} is yes for all graphs $A$ with $\chi_c(A)=\chi_c^{\bullet}(A)$. For other graphs $A$, we believe~\cref{q:cart_construction} to be difficult. It does not seem to be feasible to emulate the proofs from~\cite{BJKM06,KMSS23}.\\

With Theorems~\ref{thm:treedepth} and~\ref{thm:pathwidth}, we have determined the optimal bounds in terms of treedepth and pathwidth. These are tight due to for instance complete bipartite graphs. The famous cousin of treedepth and pathwidth is treewidth. 
In~\cite{BMS22} and~\cite{CCZ23} it was shown that $\chi_{\ell}^{\bullet}(G)\leq \mathrm{tw}(G)+1$ and $\chi_c^{\bullet}(G)\leq \mathrm{tw}(G)+2$ when $\mathrm{tw}(G) = 2$, and that this is sharp.
For this work we have verified that there is a graph $G$ with $\mathrm{tw}(G)=3$ and $\chi_c^{\bullet}(G)\geq 5$, using a linear program with as variables weights in the interval $[0,1]$ that are assigned to the correspondence-colourings of a $4$-fold cover of $K_4$, and with constraints based on Hall's Matching's theorem.
The example could subsequently be reduced to a smaller example $G$ of order $14$, for which the verification could be done immediately. The latter can be found in~\cref{app:progress_upperbounds_chicbullet_intermsof_pwtw}.

\begin{q}\label{q:upperbounds_chicbullet_intermsof_pwtw}
 Is it true that for every $t\geq 2$, there exists a graph $G$ with treewidth $t$ and $\chi_c^{\bullet}(G)= t+2$?
 Is $\chi_c^{\bullet}(G)\le \mathrm{tw}(G)+2$ for every graph $G$? Is $\chi_\ell^{\bullet}(G)\le \mathrm{tw}(G)+1$ for every graph $G$?
\end{q}

A word of caution: 
while we established  $\chi_c^{\bullet}\leq \mathrm{pw}+1$,
we have not yet excluded the possibility that there exist graphs $G$ with arbitrarily large treewidth and $\chi_c^{\bullet}(G)$ as large as $2\cdot \mathrm{tw}(G)$.
Recall that the suggested upper bound in~\cref{q:upperbounds_chicbullet_intermsof_pwtw} is false for the (non-fractional) correspondence packing number, as there exist $d$-degenerate complete bipartite graphs with $\chi_c^{\star}(G)=2d=2\cdot \mathrm{tw}=2\cdot \mathrm{pw}=2\cdot (\mathrm{td}-1)$. \\

Since the width-parameters are usually associated with complexity questions, it is natural and interesting to pose the question for the various packing numbers as well.

\begin{q}\label{q:tw&compl}
Fix $t>1$. Is the problem of determining the value of $\chi^\star_\ell(G)$ or $\chi^\bullet_\ell(G)$, given a graph $G$ of treewidth at most $t$ as input, in $P$?
\end{q}

This is true for $\chi$ and $\chi_{\ell}$~\cite{FFLRSST11,AP89}, but as alluded to in the introduction, the (fractional) packing numbers exhibit algorithmic oddities and challenges. One can ask the same for treedepth or pathwidth instead of treewidth. \\

\textbf{Note added}\\
After the acceptance of this manuscript, Kashima, Maezawa and Zhu~\cite{KMZ26} showed that for every $t\ge 2$ there is a graph $G$ with treewidth $t$ and $\chi_{\ell}^{\star}(G) \ge t+2$. They also resolved Question~\ref{q:tw&compl} for $\chi_{\ell}^{\star}(G)$.

\section*{Acknowledgement}

We thank Ross J. Kang for discussions.

\bibliographystyle{habbrv}
\bibliography{listpack}

\section*{Appendix}\label{sec: appendix}
\appendix
\section{The hypercube $Q_3$}\label{sec:appendixQ3}

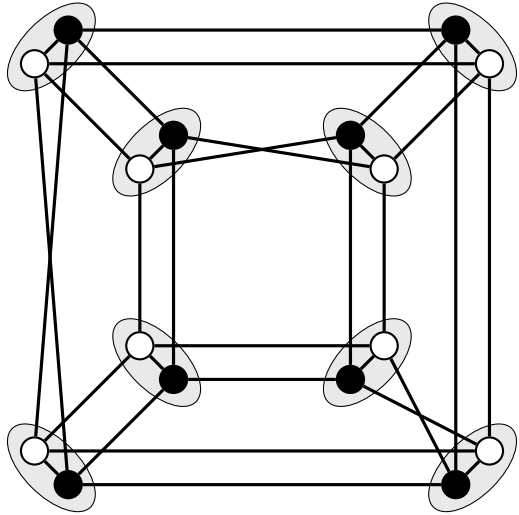
\begin{figure}[H]
\centering
\tikzmath{\scale=0.7; \dt = \scale* 0.32;\x= \scale* 2; \ellX=\scale* 30; \ellY=\scale* 15;}
\begin{tikzpicture}

\draw[fill=gray!17, rotate around={-45:(\x,\x)}] (\x,\x) ellipse (\ellX pt and \ellY pt);
\draw[fill=gray!17, rotate around={45:(\x,-\x)}] (\x,-\x) ellipse (\ellX pt and \ellY pt);
\draw[fill=gray!17, rotate around={135:(-\x,-\x)}] (-\x,-\x) ellipse (\ellX pt and \ellY pt);
\draw[fill=gray!17, rotate around={-135:(-\x,\x)}] (-\x,\x) ellipse (\ellX pt and \ellY pt);
\draw[fill=gray!17, rotate around={-45:(2*\x,2*\x)}] (2*\x,2*\x) ellipse (\ellX pt and \ellY pt);
\draw[fill=gray!17, rotate around={45:(2*\x,-2*\x)}] (2*\x,-2*\x) ellipse (\ellX pt and \ellY pt);
\draw[fill=gray!17, rotate around={135:(-2*\x,-2*\x)}] (-2*\x,-2*\x) ellipse (\ellX pt and \ellY pt);
\draw[fill=gray!17, rotate around={-135:(-2*\x,2*\x)}] (-2*\x,2*\x) ellipse (\ellX pt and \ellY pt);

\definecolor{weight0}{rgb}{1,0.4,0}
\definecolor{weight1}{rgb}{1,1,1}
\definecolor{weight2}{rgb}{0,0,0}
\begin{scope}[every node/.style={circle,thick,draw}]
\node (00) [fill=weight1] at (\x+\dt,\x-\dt) {};
\node (01) [fill=weight2] at (\x-\dt,\x+\dt) {};
\node (10) [fill=weight1] at (\x+\dt,-\x+\dt) {};
\node (11) [fill=weight2] at (\x-\dt,-\x-\dt) {};
\node (20) [fill=weight1] at (-\x-\dt,-\x+\dt) {}; 
\node (21) [fill=weight2] at (-\x+\dt,-\x-\dt) {}; 
\node (30) [fill=weight1] at (-\x-\dt,\x-\dt) {}; 
\node (31) [fill=weight2] at (-\x+\dt,\x+\dt) {}; 

\node (40) [fill=weight1] at (2*\x+\dt,2*\x-\dt) {};
\node (41) [fill=weight2] at (2*\x-\dt,2*\x+\dt) {};
\node (50) [fill=weight1] at (2*\x+\dt,-2*\x+\dt) {};
\node (51) [fill=weight2] at (2*\x-\dt,-2*\x-\dt) {};
\node (60) [fill=weight1] at (-2*\x-\dt,-2*\x+\dt) {}; 
\node (61) [fill=weight2] at (-2*\x+\dt,-2*\x-\dt) {}; 
\node (70) [fill=weight1] at (-2*\x-\dt,2*\x-\dt) {}; 
\node (71) [fill=weight2] at (-2*\x+\dt,2*\x+\dt) {};

\end{scope}

\begin{scope}[
 every node/.style={fill=white,circle},
 every edge/.style={draw=black,very thick}]
 \path (00) edge (01); \path (10) edge (11); \path (20) edge (21); \path (30) edge (31); \path (40) edge (41); \path (50) edge (51); \path (60) edge (61); \path (70) edge (71); 

\path (00) edge (10);
\path (01) edge (11);
\path (10) edge (20);
\path (11) edge (21);
\path (20) edge (30);
\path (21) edge (31);
\path (30) edge (01);
\path (31) edge (00);

\path (40) edge (50);
\path (41) edge (51);
\path (50) edge (60);
\path (51) edge (61);
\path (60) edge (71);
\path (61) edge (70);
\path (70) edge (40);
\path (71) edge (41);

\path (00) edge (40);
\path (01) edge (41);
\path (10) edge (51);
\path (11) edge (50);
\path (20) edge (60);
\path (21) edge (61);
\path (30) edge (70);
\path (31) edge (71);
\end{scope}
\end{tikzpicture}
 \caption{A $2$-fold correspondence-cover $(H,L)$ of the hypercube $Q_3$, demonstrating that $\chi_{c}^{\bullet}(Q_3) >3$. 
 For each vertex $v$ of $Q_3$, its list $L(v)=\{1_v,2_v\}$ is indicated by a grey ellipse surrounding a white dot representing $1_v$ and a black dot representing $2_v$.
 For most edges $uv$ of $Q_3$, the matching between $L(u)$ and $L(v)$ is the `identity matching', consisting of edges $1_u 1_v$ and $2_u 2_v$ in $H$. The exceptions are formed by a matching of $Q_3$ that consists of three \emph{special edges} $uv$ for which the matching between $L(u)$ and $L(v)$ is `crossing', i.e., consisting of $1_u 2_v$ and $2_u 1_v$. The special edges are chosen such that every induced cycle $C$ of $Q_3$ traverses an odd number of them. Therefore the restriction of $(H,L)$ to an induced cycle $C$ of $Q_3$ cannot contain any independent set on $|V(C)|$ vertices. 
 From this it easily follows that every independent set of the $16$-vertex graph $H$ has size at most $5$. Hence the fractional chromatic number of $H$ is at least $16/5$, which is strictly larger than three.
}
 \label{fig:hypercube_cover}
\end{figure}

\section{A treewidth three graph with $\chi_c^{\bullet}=5$. Progress for~\cref{q:upperbounds_chicbullet_intermsof_pwtw}}\label{app:progress_upperbounds_chicbullet_intermsof_pwtw}

A subcover $(L,H)$ of a $4$-fold cover of a $3$-tree $G$ with $\chi_c^\bullet(G) \ge 5$ is presented in~\cref{fig:treewidth3_chicbullet5} and~\cref{fig:treewidth3_chicbullet5_cover}. Note that $G$ has treewidth three, so this constitutes some progress towards~\cref{q:upperbounds_chicbullet_intermsof_pwtw}. 
The cover graph $H$ has fractional chromatic number $4+1/2092$. Since this is strictly larger than $4$, it implies that $\chi_c^\bullet(G) \ge 5$. We also list the edges of $H$, so that interested readers can verify the statement with a simple computer check. There are $45$ vertices and the edges are described by the following list of adjacencies, where for example a:[b c d] means that vertex $a$ is adjacent to $b,c$ and $d$.\\
1: [2 3 4 6 9 13 23 29 32 35 43],
2: [3 4 7 10 14 17 24 30 33 37],
3: [4 5 11 15 18 38 44],
4: [8 12 16 19 25 31 34 36 39 45],
5: [6 7 8 9 13 20 24 30 33 37 44],
6: [7 8 10 14 23 26 29 32 38 40],
7: [8 11 15 21 27 41 43],
8: [12 16 22 25 28 31 34 39 42 45],
9: [10 11 12 13 17 26 37 41 44],
10: [11 12 14 20 30 35 43],
11: [12 15 18 21 27 29 38 40],
12: [16 19 22 28 31 36 39 42 45],
13: [14 15 16 17 24 26 41],
14: [15 16 20 27 33 40],
15: [16 18 21 23 32 35],
16: [19 22 25 28 34 36 42],
17: [18 19],
18: [19],
20: [21 22],
21: [22],
23: [24 25],
24: [25],
26: [27 28],
27: [28],
29: [30 31],
30: [31],
32: [33 34],
33: [34],
35: [36],
37: [38 39],
38: [39],
40: [41 42],
41: [42],
43: [44 45],
44: [45].
\\


A fractional clique certifying that $\chi_c^{\bullet}(G)\geq 5$ is given by the following vertex weights. The sum of all the weights is $4+1/2092$, while on each independent set the sum of the weights is at most one. Here for example a:[b c d] means that weight $a/4148$ is assigned to each of the vertices $b,c$ and $d$. \\
2:  [26 27 28],         
18: [36 37],           
28: [40 41 42],        
32: [43],
70: [37 38 39],
86: [44 45],
102: [23],
168: [25 26],
200: [33 34 35],
282: [29],
293: [20 21 22],
310: [30 31],
323: [17 18 19],
447: [9],
491: [13],
493: [10],
511: [14],
614: [7],
689: [5],
700: [1 3 6],
703: [2],
786: [11],
804: [15],
1182: [8],
1190: [16],
1200: [4 12].

\begin{figure}[ht]
\centering
 \begin{minipage}{\linewidth}
 \centering
 \begin{minipage}{0.475\linewidth}
 \centering
 \tikzmath{\scale=0.475; \dt = \scale* 1;\x= \scale* 2; \ex= 2.5*\x; \ellX=\scale* 30; \ellY=\scale* 15;}
 \begin{tikzpicture} [rotate=0]
 \definecolor{weight0}{rgb}{1,0.4,0}
 \definecolor{weight1}{rgb}{1,1,1}
 \definecolor{weight2}{rgb}{0,0,0}
 \begin{scope}[every node/.style={circle,thick,draw}]
 \node (0) [fill=weight2] at (\x,\x) {};
 \node (1) [fill=weight2] at (-\x,\x) {};
 \node (2) [fill=weight2] at (-\x,-\x) {};
 \node (3) [fill=weight2] at (\x,-\x) {};
 
 \node (4) [fill=weight2] at (\ex+\dt,\ex-\dt) {};
 \node (5) [fill=weight2] at (\ex,\ex) {};
 \node (6) [fill=weight2] at (\ex-\dt,\ex+\dt) {};
 
 \node (7) [fill=weight2] at (-\ex-\dt,\ex-\dt) {};
 \node (8) [fill=weight2] at (-\ex+\dt,\ex+\dt) {};

 \node (9) [fill=weight2] at (-\ex-\dt,-\ex+\dt) {};
 \node (10) [fill=weight2] at (-\ex,-\ex) {};
 \node (11) [fill=weight2] at (-\ex+\dt,-\ex-\dt) {};
 
 \node (12) [fill=weight2] at (\ex+\dt,-\ex+\dt) {};
 \node (13) [fill=weight2] at (\ex-\dt,-\ex-\dt) {};

 \end{scope}

 \begin{scope}[
 every node/.style={fill=white,circle},
 every edge/.style={draw=black,very thick}]
 \path (0) edge (1);
 \path (1) edge (2);
 \path (2) edge (3);
 \path (3) edge (0);
 \path(0) edge (2);
 \path(1) edge (3);

 every edge/.style={draw=black, thick}]
 \path (4) edge (0); \path (4) edge (1); \path (4) edge (3);
 \path (5) edge (0); \path (5) edge (1); \path (5) edge (3);
 \path (6) edge (0); \path (6) edge (1); \path (6) edge (3);
 \path (7) edge (0); \path (7) edge (1); \path (7) edge (2);
 \path (8) edge (0); \path (8) edge (1); \path (8) edge (2);
 \path (9) edge (1); \path (9) edge (2); \path (9) edge (3);
 \path (10) edge (1); \path (10) edge (2); \path (10) edge (3);
 \path (11) edge (1); \path (11) edge (2); \path (11) edge (3);
 \path (12) edge (0); \path (12) edge (2); \path (12) edge (3);
 \path (13) edge (0); \path (13) edge (2); \path (13) edge (3);
 
 \end{scope}
 \end{tikzpicture}
 \end{minipage}
 \centering
\end{minipage}
\caption{A graph $G$ with treewidth three and $\chi_c^{\bullet}(G) =5$.
 }\label{fig:treewidth3_chicbullet5}
\end{figure}
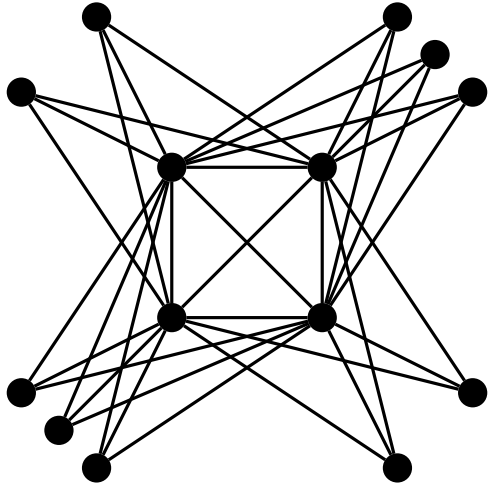

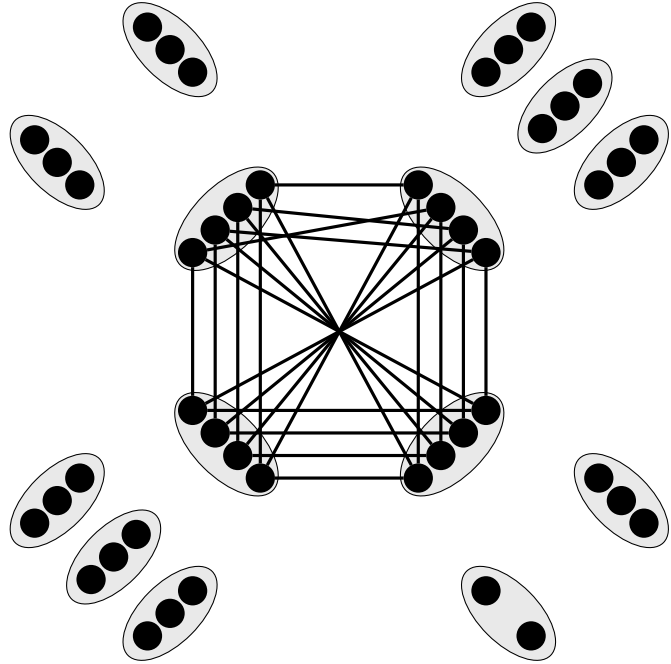
\begin{figure}[ht]
\centering
 \begin{minipage}{0.475\linewidth}
 \centering
 \tikzmath{\scale=0.75; \dt = \scale* 0.2;\x= \scale* 2; \b=\scale *1;\ellX=\scale* 30; \ellY=\scale* 15;}
 \begin{tikzpicture} [rotate=0]
 
 \draw[fill=gray!17, rotate around={-45:(\x,\x)}] (\x,\x) ellipse (1.1*\ellX pt and 1.1*\ellY pt);
 \draw[fill=gray!17, rotate around={45:(\x,-\x)}] (\x,-\x) ellipse (1.1*\ellX pt and 1.1*\ellY pt);
 \draw[fill=gray!17, rotate around={135:(-\x,-\x)}] (-\x,-\x) ellipse (1.1*\ellX pt and 1.1*\ellY pt);
 \draw[fill=gray!17, rotate around={-135:(-\x,\x)}] (-\x,\x) ellipse (1.1*\ellX pt and 1.1*\ellY pt);
 
 \draw[fill=gray!17, rotate around={45:(2*\x,2*\x)}] (2*\x,2*\x) ellipse (\ellX pt and \ellY pt);
 \draw[fill=gray!17, rotate around={45:(2*\x+\b,2*\x-\b)}] (2*\x+\b,2*\x-\b) ellipse (\ellX pt and \ellY pt);
 \draw[fill=gray!17, rotate around={45:(2*\x-\b,2*\x+\b)}] (2*\x-\b,2*\x+\b) ellipse (\ellX pt and \ellY pt);
 
 \draw[fill=gray!17, rotate around={45:(-2*\x,-2*\x)}] (-2*\x,-2*\x) ellipse (\ellX pt and \ellY pt);
 \draw[fill=gray!17, rotate around={45:(-2*\x-\b,-2*\x+\b)}] (-2*\x-\b,-2*\x+\b) ellipse (\ellX pt and \ellY pt);
 \draw[fill=gray!17, rotate around={45:(-2*\x+\b,-2*\x-\b)}] (-2*\x+\b,-2*\x-\b) ellipse (\ellX pt and \ellY pt);

 \draw[fill=gray!17, rotate around={135:(-2*\x-\b,2*\x-\b)}] (-2*\x-\b,2*\x-\b) ellipse (\ellX pt and \ellY pt);
 \draw[fill=gray!17, rotate around={135:(-2*\x+\b,2*\x+\b)}] (-2*\x+\b,2*\x+\b) ellipse (\ellX pt and \ellY pt);
 
 \draw[fill=gray!17, rotate around={135:(2*\x+\b,-2*\x+\b)}] (2*\x+\b,-2*\x+\b) ellipse (\ellX pt and \ellY pt);
 \draw[fill=gray!17, rotate around={135:(2*\x-\b,-2*\x-\b)}] (2*\x-\b,-2*\x-\b) ellipse (\ellX pt and \ellY pt);
 
 \definecolor{weight0}{rgb}{1,0.4,0}
 \definecolor{weight1}{rgb}{1,1,1}
 \definecolor{weight2}{rgb}{0,0,0}
 \begin{scope}[every node/.style={circle,thick,draw}]
 
 \node(0) [fill=weight2] at (\x+3*\dt, \x-3*\dt) {};
 \node(1) [fill=weight2] at (\x+\dt, \x-\dt) {};
 \node(2) [fill=weight2] at (\x-\dt,\x+\dt) {};
 \node(3) [fill=weight2] at (\x-3*\dt,\x+3*\dt) {};
 \node(4) [fill=weight2] at (-\x-3*\dt, \x-3*\dt) {};
 \node(5) [fill=weight2] at (-\x-\dt, \x-\dt) {};
 \node(6) [fill=weight2] at (-\x+\dt,\x+\dt) {};
 \node(7) [fill=weight2] at (-\x+3*\dt,\x+3*\dt) {};
 
 \node(8) [fill=weight2] at (\x+3*\dt, -\x+3*\dt) {};
 \node(9) [fill=weight2] at (\x+\dt, -\x+\dt) {};
 \node(10) [fill=weight2] at (\x-\dt,-\x-\dt) {};
 \node(11) [fill=weight2] at (\x-3*\dt,-\x-3*\dt) {};
 \node(12) [fill=weight2] at (-\x-3*\dt, -\x+3*\dt) {};
 \node(13) [fill=weight2] at (-\x-\dt, -\x+\dt) {};
 \node(14) [fill=weight2] at (-\x+\dt,-\x-\dt) {};
 \node(15) [fill=weight2] at (-\x+3*\dt,-\x-3*\dt) {};
 
 \node(16) [fill=weight2] at (2*\x-2*\dt, 2*\x-2*\dt) {};
 \node(17) [fill=weight2] at (2*\x, 2*\x) {};
 \node(18) [fill=weight2] at (2*\x+2*\dt, 2*\x+2*\dt) {};
 \node(19) [fill=weight2] at (2*\x-2*\dt-\b, 2*\x-2*\dt+\b) {};
 \node(20) [fill=weight2] at (2*\x-\b, 2*\x+\b) {};
 \node(21) [fill=weight2] at (2*\x+2*\dt-\b, 2*\x+2*\dt+\b) {};
 \node(22) [fill=weight2] at (2*\x-2*\dt+\b, 2*\x-2*\dt-\b) {};
 \node(23) [fill=weight2] at (2*\x+\b, 2*\x-\b) {};
 \node(24) [fill=weight2] at (2*\x+2*\dt+\b, 2*\x+2*\dt-\b) {};

 \node(25) [fill=weight2] at (-2*\x+2*\dt, -2*\x+2*\dt) {};
 \node(26) [fill=weight2] at (-2*\x, -2*\x) {};
 \node(27) [fill=weight2] at (-2*\x-2*\dt, -2*\x-2*\dt) {};
 \node(28) [fill=weight2] at (-2*\x+2*\dt+\b, -2*\x+2*\dt-\b) {};
 \node(29) [fill=weight2] at (-2*\x+\b, -2*\x-\b) {};
 \node(30) [fill=weight2] at (-2*\x-2*\dt+\b, -2*\x-2*\dt-\b) {};
 \node(31) [fill=weight2] at (-2*\x+2*\dt-\b, -2*\x+2*\dt+\b) {};
 \node(32) [fill=weight2] at (-2*\x-\b, -2*\x+\b) {};
 \node(33) [fill=weight2] at (-2*\x-2*\dt-\b, -2*\x-2*\dt+\b) {};

 \node(34) [fill=weight2] at (-2*\x+2*\dt-\b, 2*\x-2*\dt-\b) {};
 \node(35) [fill=weight2] at (-2*\x-\b, 2*\x-\b) {};
 \node(36) [fill=weight2] at (-2*\x-2*\dt-\b, 2*\x+2*\dt-\b) {};
 \node(37) [fill=weight2] at (-2*\x+2*\dt+\b, 2*\x-2*\dt+\b) {};
 \node(38) [fill=weight2] at (-2*\x+\b, 2*\x+\b) {};
 \node(39) [fill=weight2] at (-2*\x-2*\dt+\b, 2*\x+2*\dt+\b) {};
 
 \node(34) [fill=weight2] at (2*\x-2*\dt+\b, -2*\x+2*\dt+\b) {};
 \node(35) [fill=weight2] at (2*\x+\b, -2*\x+\b) {};
 \node(36) [fill=weight2] at (2*\x+2*\dt+\b, -2*\x-2*\dt+\b) {};
 \node(37) [fill=weight2] at (2*\x-2*\dt-\b, -2*\x+2*\dt-\b) {};
 \node(39) [fill=weight2] at (2*\x+2*\dt-\b, -2*\x-2*\dt-\b) {};
 \end{scope}
 
 \begin{scope}[
 every node/.style={fill=white,circle},
 every edge/.style={draw=black,very thick}]
 \path (0) edge (5);
 \path (0) edge (8);
 \path (0) edge (12);
 \path (1) edge (6);
 \path (1) edge (9);
 \path (1) edge (13);
 \path (2) edge (4);
 \path (2) edge (10);
 \path (2) edge (14);
 \path (3) edge (7); 
 \path (3) edge (11);
 \path (3) edge (15);
 \path (4) edge (8);
 \path (4) edge (12);
 \path (5) edge (9);
 \path (5) edge (13);
 \path (6) edge (10);
 \path (6) edge (14);
 \path (7) edge (11);
 \path (7) edge (15);
 \path (8) edge (12);
 \path (9) edge (13);
 \path (10) edge (14);
 \path (11) edge (15);
 \end{scope}
 \end{tikzpicture} 
 \end{minipage}

\caption{A correspondence-cover $(L,H)$ of the graph $G$ in~\cref{fig:treewidth3_chicbullet5}. The covergraph $H$ has fractional chromatic number $>4$ and all lists of size $\leq 4$, thus certifying that $\chi_c^{\bullet}(G) >4$. The grey ellipses indicate the lists. To avoid clutter, only the matchings on the central $K_4$ of $G$ are depicted.
 }\label{fig:treewidth3_chicbullet5_cover}
\end{figure}

\end{document}